\documentclass{m2an}
\usepackage{graphicx} 
\usepackage{amsmath,amsfonts,eucal}
\usepackage{amssymb,amsbsy,mathrsfs,amscd}
\usepackage{color}
\usepackage{geometry}
\usepackage{mathtools}
\usepackage{accents}
\usepackage{amsfonts,bm}
\usepackage[all,cmtip]{xy}

\usepackage{tikz-cd}
\usetikzlibrary{matrix}
\usetikzlibrary{arrows,shapes}
\usetikzlibrary{arrows, decorations.markings,fit}
\usetikzlibrary{calc,3d}
\usepackage{tikz-3dplot}
\usepackage{pgfplots}
\allowdisplaybreaks[4]

\newcommand{\bu}{\bm u}

\newcommand{\bv}{\bm v}

\newcommand{\bbR}{\mathbb{R}}
\newcommand{\nab}{\nabla}

\newcommand{\p}{\partial}

\newcommand{\bH}{\bm H}

\newcommand{\bZ}{\bm Z}
\newcommand{\bz}{\bm z}
\newcommand{\bV}{\bm V}

\newcommand{\calT}{\mathcal{T}}

\newcommand{\bw}{\bm w}

\newcommand{\bcurl}{{\bf curl}}
\newcommand{\Div}{{\rm div}}
\newcommand{\mct}{\mathcal{T}_h}
\newcommand{\mcf}{\mathcal{F}_h}

\newcommand{\bn}{{\bm n}}
\newcommand{\pol}{\EuScript{P}}
\newcommand{\bpol}{\boldsymbol{\pol}}

\newcommand{\calS}{\mathcal{S}}
\newcommand{\calF}{\mathcal{F}}

\numberwithin{equation}{section}

\newtheorem{theorem}{Theorem}[section]
\newtheorem{lemma}[theorem]{Lemma}

\newtheorem{corollary}[theorem]{Corollary}
\newtheorem{proposition}[theorem]{Proposition}
\theoremstyle{definition}

\theoremstyle{remark}
\newtheorem{remark}[theorem]{Remark}

\newcommand{\be}{{\bm e}}

\newcommand{\bM}{{\bm M}}
\newcommand{\bkappa}{{\bm \kappa}}
\newcommand{\bPi}{{\bm \Pi}}
\newcommand{\bI}{{\bm I}}
\newcommand{\bpsi}{{\bm \psi}}

\def\jump#1{[\hspace{-1.5pt}[#1]\hspace{-1.5pt}]}

\newcommand{\rev}[1]{{ #1}}

\begin{document}

\title{A CutFEM divergence--free discretization for the Stokes problem}

\author{Haoran Liu}\address{Department of Mathematics, University of Pittsburgh, Pittsburgh, PA 15260. HAL104@pitt.edu} 
\author{Michael Neilan}\address{Department of Mathematics, University of Pittsburgh, Pittsburgh, PA 15260. neilan@pitt.edu} 
\author{Maxim Olshanskii}\address{Department of Mathematics, University of Houston, Houston, TX 77204. maolshanskiy@uh.edu}

\thanks{The first two authors were supported in part by NSF grant no. DMS-2011733. The third author was supported 
in part by NSF grants  DMS-1953535 and DMS-2011444.}

\thispagestyle{empty}

\begin{abstract}
We construct and analyze a CutFEM discretization
for the Stokes problem based on the Scott-Vogelius pair.
The discrete piecewise polynomial spaces
are defined on macro-element triangulations which are not fitted
to the smooth physical domain.  
 Boundary conditions are imposed via penalization through
 the help of a Nitsche-type discretization, whereas stability
 with respect to small and anisotropic cuts of the bulk elements is ensured by adding local 
 ghost penalty stabilization terms.  We show stability
 of the scheme as well as a divergence--free property of the 
discrete velocity outside an $O(h)$ neighborhood of the boundary.
To mitigate the error caused by the violation of the divergence-free condition, we introduce local grad-div stabilization. 
The error analysis shows that the grad-div parameter can scale like $O(h^{-1})$, allowing a rather heavy penalty for the violation of  
mass conservation, while still ensuring optimal order error estimates.
\end{abstract}

\subjclass{65N30,65N12,76D07}
\keywords{CutFEM, Stokes,divergence--free, convergence analysis}

\maketitle

\section{Introduction}
Originating from first principles,  mass conservation is a favorable property for numerical solutions to equations governing the motion of fluids.  For incompressible viscous fluids there is a recent effort in developing Galerkin  methods that conserve the mass through enforcing the discrete velocity to be 
divergence--free~\cite{Evans1,Falk1,Guz1,Guz2,Lehr1,Zhang}. In many cases however, the exactly divergence--free approximations come at a price of high-order spaces or confined to certain triangulations.
This extra complexity  makes it difficult, if at all possible, to apply them for modeling of flows with additional features such as propagating interfaces and free surfaces exhibiting large deformations.
One way to address the difficulty of computing flows with finite elements in heavily deforming volumes is to uncouple the volume triangulation from interface/boundary tracking. For sharp interface representations (e.g. by a level-set method~\cite{Sethian1} in contrast to  diffuse-interface approach~\cite{Ander1}) this suggests using geometrically unfitted finite element (FE) methods, e.g. XFEM~\cite{XFEM} or cutFEM~\cite{cutFEM}. While unfitted variants of  equal order pressure-velocity, Taylor--Hood and several other well-known finite element methods were studied recently~\cite{St1,BurmanHansbo14,GuzmanMaxim18,St2,St3,St4}, unfitted mass-conserving elements for primitive-variable formulations of the Stokes or incompressible Navier-Stokes equations are not well developed yet. 

This brings our attention to the Scott-Vogelius (SV) finite element pair~\cite{SV1} of continuous polynomial elements of degree $k$ for velocity and discontinuous polynomial elements of degree $(k-1)$ for pressure.  The SV element is known to be inf-sup stable~\cite{GR} for shape-regular triangulations of $\mathbb{R}^2$ and 
$k\ge 4$ provided the triangulation does not contain nearly singular vertices \cite{SV1,GS}.
For $\mathbb{R}^3$ the situation is more subtle and for a specific structured family of triangulations,
stability is known for $k\ge6$~\cite{Zhang2}. 
On triangulations resulting from a barycenteric (or Clough-Tocher/Alfeld) refinement of a macro triangulation the stability, however, follows in $\mathbb{R}^d$ for $k\ge d$~\cite{ArnoldQiu92,Zhang,Guz2}.  

Since the divergence of any FE velocity field belong to the pressure FE space, a standard mixed formulation of the Stokes problem using the SV element delivers pointwise divergence-free solutions, providing us with an example of a stable and relatively simple mass conserving finite element method.

A geometrically unfitted variant of the SV finite element in two dimensions was considered for the first time in~\cite{LNO}. In that paper, the authors 
constructed a boundary correction scheme, where the Stokes problem is solved
on a strict interior domain, and boundary data is transferred, within a Nitsche-type formulation, via a Taylor expansion.
Here we take a different approach: A boundary or interface condition on a surface cutting through a background mesh is imposed with the help of the Nitsche method~\cite{hansbo2002unfitted}, while stability with respect to small and anisotropic cuts of the bulk elements is ensured by adding local ghost penalty stabilization terms. This renders our approach as the cut-Scott-Vogelius finite element method.  The stability of pressure for cut elements requires an additional ghost penalty term defined on a thin strip of elements in a proximity of the boundary. Since this additional stabilizing term alters the divergence-condition in the mixed formulation, the resulting velocity is not divergence--free, strictly speaking. Nevertheless, we show that  pointwise mass conservation holds in the volume occupied by the fluid except in an $O(h)$ strip. Furthermore, to minimize the error caused by the violation of the divergence--free condition, we introduce local grad-div stabilization~\cite{Olsh1}. Our analysis reveals that the grad-div parameter can scale with $h^{-1}$, allowing a rather heavy penalty for the violation of  $\Div\bu=0$ in the strip and still ensuring an optimal order error estimate.

The remainder of paper  is organized as follows.  Section~\ref{sec:Prelim} is preparatory and collects necessary preliminaries on meshing and FE spaces definitions.         It proceeds with the formulation of finite element method and the proof of an internal mass-conservation property. Section~\ref{sec:Stab} addresses stability of the finite element formulation.  In section~\ref{sec:Conv} we prove our main result about the convergence of the method. 
\rev{We extend our method and analysis to two-dimensional Powell-Sabin splits in section~\ref{sec:PS}, and give}
results of a few numerical examples in section~\ref{sec:Num}.

\section{Preliminaries and Finite Element Method}\label{sec:Prelim}

Consider the Stokes problem on a bounded, contractible, and open domain $\Omega\subset \bbR^d$ with smooth boundary:
\begin{subequations}
		\label{eqn:Stokes}
	\begin{alignat}{2}
		\label{eqn:Stokes1}
		-{\Delta \bu} + \nab p & = {\bm f}\qquad &&\text{in }\Omega,\\
				\label{eqn:Stokes2}
		\Div \bu & = 0\qquad &&\text{in }\Omega,\\
		\bu& = {0}\qquad &&\text{on }\Gamma:=\p \Omega.
	\end{alignat}
\end{subequations}
To formulate an unfitted finite element method for \eqref{eqn:Stokes} we embed $\Omega$ into an open, polytopal domain  $S$, i.e.
$\bar{\Omega}\subset S$, and let
$\mct$ be a simplicial shape-regular triangulation of $S$.
For simplicity, we assume that $\calT_h$ is quasi-uniform.
We denote by $\mct^{ct}$ the resulting Clough-Tocher/Alfeld (CT)  triangulation
obtained by connecting the vertices of each element in $\mct$
with its barycenter \cite{LaiSchumakerBook,FGN20}.
Let 
\begin{equation}\label{eqn:mctiDef}
\mct^i:=\{T\in \mct: T\subset \Omega\},\qquad \Omega_h^i =  {\rm Int}\Big(\bigcup_{T\in \mct^i} \bar T\Big)
\end{equation}
be the set of interior simplices and interior domain, respectively.
We then define the analogous set with respect to the CT refinement (cf.~Figure \ref{fig:Meshes}):
\[
\mct^{ct,i}:=\{K\in \mct^{ct}:\ K\subset T,\ \exists T\in \mct^i\}.
\]

\begin{figure}[h]
\centering
\includegraphics[scale=0.5]{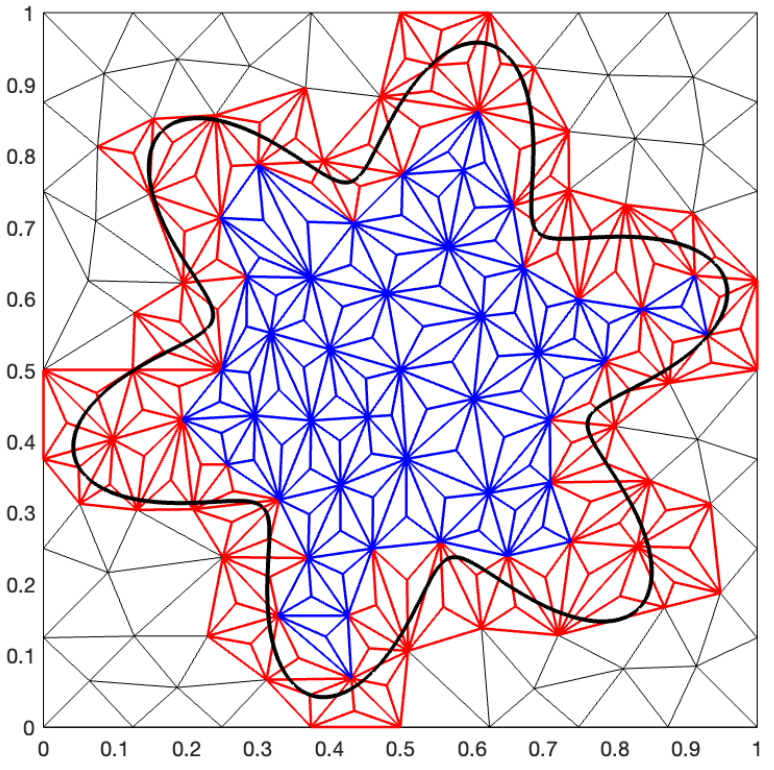}
\includegraphics[scale=0.5]{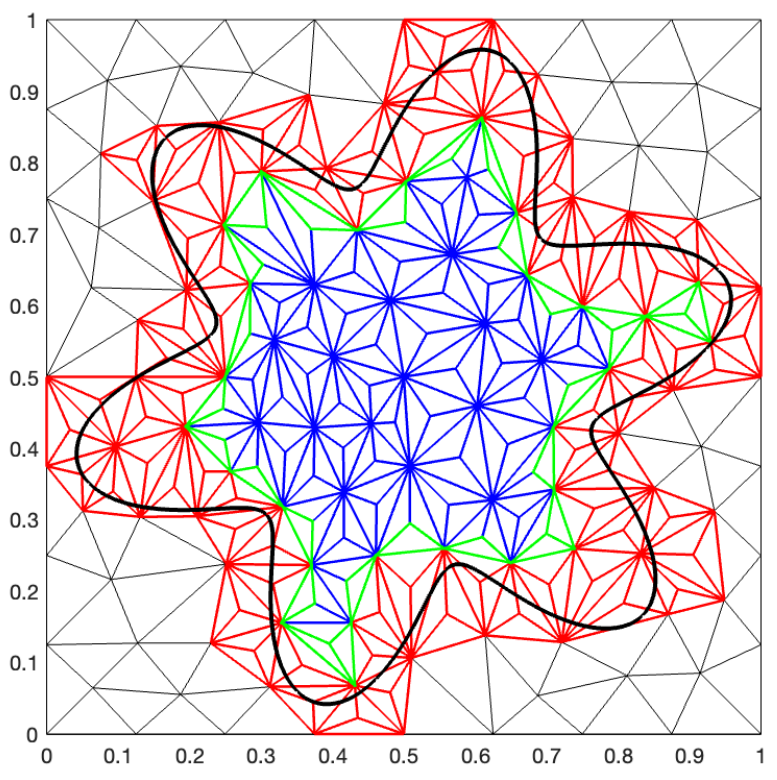}
\caption{\label{fig:Meshes}Left: Example of an interior Clough-Tocher mesh $\calT_h^{ct,i}$ (blue)
and boundary Clough-Tocher mesh $\calT_h^{ct,\Gamma}$ (red).
Right: The same mesh, but with $\widetilde \calT^{ct,\Gamma}_h\backslash \calT_h^{ct,\Gamma}$ (green)
and $\widetilde{\mathcal{T}}_h^{ct,i}$ (blue).
}
\end{figure}

\begin{remark} 
Note that
\[
 \Omega_h^i =  {\rm Int}\Big(\bigcup_{K\in \mct^{ct,i}} \bar K\Big),
\]
since we include in $\mct^{ct,i}$ only those micro-triangles which belong to  internal macro-triangles, there can be $K\in \mct^{ct}$ which are strictly  inside $\Omega$ but not in $\mct^{ct,i}$, i.e., the inclusion
 $
 \mct^{ct,i}\subset \{K\in \mct^{ct}: K\subset \Omega\}
$
 is generally strict.
\end{remark}

On the interior domain, we define $\mcf^i$ to be
the set of $(d-1)$-dimensional interior faces of the unrefined triangulation $\mct^i$; 
that is, $F\in \mcf^i$ provided there exists two distinct simplices $T_1,T_2\in \mct^i$
such that $F = \p T_1\cap \p T_2$.
We also let
\[
\mct^{\Gamma}:=\{T\in \mct: {\rm meas}_{d-1}(T\cap \Gamma)>0\},\qquad \Omega_h^\Gamma =  {\rm Int}\Big(\bigcup_{T\in \mct^\Gamma} \bar T\Big)
\]
to be the set of simplices that cut through the interface $\Gamma$ and the corresponding domain, respectively.

Define
\[
\mct^{e}:=\{T\in \mct: T\in\mct^i\ {\rm or}\ T\in\mct^\Gamma \},\qquad \Omega_h^e =  {\rm Int}\Big(\bigcup_{T\in \mct^e} \bar T\Big)
\]
to be the set of triangles that either are interior or cut through the interface $\Gamma$ and the corresponding domain, respectively.
We refer to these quantities as the exterior triangulation and exterior domain, respectively.
The analogous sets with respect to the CT refinement are given by
\[
\mct^{ct,\Gamma}:=\{K\in \mct^{ct}:\ K\subset T,\ \exists T\in \mct^{\Gamma}\},
\]
\[
\mct^{ct,e}:=\{K\in \mct^{ct}:\ K\subset T,\ \exists T\in \mct^{e}\}.
\]

Note that now we have
\[
\Omega_h^\Gamma =  {\rm Int}\Big(\bigcup_{K\in \mct^{ct,\Gamma}} \bar K\Big),\quad \text{and}\quad
\Omega_h^e =  {\rm Int}\Big(\bigcup_{K\in \mct^{ct,e}} \bar K\Big).
\]
\begin{remark}
Note that, in general, there exists
$K\in \mct^{ct,\Gamma}$ such that $\bar K\cap \bar \Omega = \emptyset$.
Consequently, in the finite element method presented below, 
there exists active basis functions with support strictly outside the physical domain $\Omega$.
\end{remark}

We define the sets of faces:
\begin{align*}
\mcf^\Gamma&:=\{F:\ F\ {\rm is\ a\ face\ in}\ \mct^\Gamma,\ F\not\subset\partial\Omega_h^e\},\\
\mcf^e&:=\{F:\ F\ {\rm is\ an\ interior\ face\ in}\ \mct^e\},\\
\mcf^{ct,\Gamma}&:=\{F:\ F\ {\rm is\ a\ face\ in}\ \mct^{ct,\Gamma},\ F\not\subset\partial\Omega_h^e\},\\
\mcf^{ct,e}&:=\{F:\ F\ {\rm is\ an\ interior\ face\ in}\ \mct^{ct,e}\}.
\end{align*}

For $K\in\mct^{ct,\Gamma}$, we define $K_\Gamma=\overline{K}\cap\Gamma$, so that $\sum_{K\in\mct^{ct,\Gamma}}|K_\Gamma|=|\Gamma|$.
For a simplex $K$, we set $h_K = {\rm diam}(h_K)$.
Note that, because of the quasi-uniformity and shape-regularity assumption, there holds
$h_K\approx h:=\max_{T\in \calT_h} h_T$ for all $K\in \calT_h^{ct,e}$
and $h_F \approx h$ for all $F\in \calF_h^{ct,e}$.
We denote by $\bn$ an outward normal of a domain which will be clear from its context.
The constant $C$ (with or without subscripts) will denote a generic positive constant that is independent of $h$,
how the boundary $\Gamma$ cuts the mesh, or any method-dependent parameters.

\rev{
For some triangulation $\mathcal{S}_h$ (e.g., $\mathcal{S}_h = \calT^{ct,e}_h$), 
we define the polynomial spaces:
\begin{alignat*}{2}
&\pol_k(\mathcal{S}_h) = \{v\in L^2(D_h):\ v|_K\in \pol_k(K)\ \forall K\in \mathcal{S}_h\},\quad &&\mathring{\pol}_k(\mathcal{S}_h) = \pol_k(\mathcal{S}_h)\cap L^2_0(D_h),\\
&\bpol_k^c(\mathcal{S}_h) = [\pol_k(\mathcal{S}_h)\cap H^1(D_h)]^d,\quad &&\mathring{\bpol}_k^c(\mathcal{S}_h) = \bpol_k^c(\mathcal{S}_h)\cap \bH^1_0(D_h),
\end{alignat*}
where $D_h ={\rm int}( \cup_{K\in \mathcal{S}_h} \bar K)$, and $\pol_k(K)$ is the space of scalar polynomials of degree $\le k$ with domain $K$.}
For an integer $k\ge d$, define the finite element spaces with respect to $\mct^{ct,e}$:
\begin{align*}
\bV_h &= \{\rev{\bv\in \bpol_k^c(\mct^{ct,e})},\ 
 \int_{\p\Omega_h^i}\bv\cdot\bn=0\},\qquad
Q_h   = \{\rev{q\in \pol_{k-1}(\mct^{ct,e})},\ \int_{\Omega_h^i}q=0\},
\end{align*}
%
and the analogous spaces with respect to the interior mesh
\begin{align*}
\rev{\bV^i_h = \mathring{\bpol}^c_k(\mct^{ct,i}),\qquad Q_h^i = \mathring{\pol}_{k-1}(\mct^{ct,i}).}
%
\end{align*}
The definitions of the spaces imply that the divergence maps $\bV_h$ into
$Q_h$.  Likewise, the divergence maps $\bV_h^i$ into $Q_h^i$.

\subsection{Finite Element Method}
Here, we will define a modified version of the finite element method given in \cite[Section 3.2]{GuzmanMaxim18}
based on the Scott--Vogelius pair.
First we define the mesh-dependent bilinear form
with grad-div stabilization:
\begin{equation}\label{eqn:ahDef}
a_h(\bu,\bv):= (\nabla \bu, \nabla \bv)+\gamma (\Div \bu, \Div \bv)+s_h(\bu,\bv)+\bm{{\rm j}}_h (\bu,\bv)+\eta j_h(\bu,\bv),
\end{equation}
where $(\cdot,\cdot)$ denotes the $L^2$ inner product over $\Omega$, 
$\gamma\ge 0$ is the grad-div parameter, and
 $\eta>0$ is a Nitsche-type penalty parameter,
\begin{align*}
s_h(\bu,\bv)&=-\int_{\Gamma}((\bn^\intercal\nabla\bu)\cdot\bv+(\bn^\intercal\nabla\bv)\cdot\bu),\\
j_h(\bu,\bv)&=\sum_{K\in \mct^{ct,\Gamma}}\frac{1}{h_K}\int_{K_\Gamma}\bu\cdot\bv,\\
\bm{{\rm j}}_h (\bu,\bv)&=\sum_{F\in\mcf^{ct,\Gamma}}\sum_{\ell=1}^{k}h_F^{2\ell-1}\int_{F}[\partial^\ell_n\bu][\partial^\ell_n\bv],
\end{align*}
where we recall $K_\Gamma = \bar K\cap \Gamma$.
Here, $\p_n^\ell \bv$ denotes the derivative of order $\ell$ of $\bv$ in the direction $\bv$,
and $[\bw]|_F$ denotes the jump of a function $\bw$ across $F$.
The continuity equations are discretized via the bilinear form
\begin{equation}\label{eqn:bDef}
b(p,\bv):=-(p,\Div \bv)+\int_\Gamma (\bv\cdot \bn)p.
\end{equation}
The finite element method reads:
find $(\bu_h,p_h)\in \bV_h\times Q_h$ such that
\begin{align}\label{fem}
\Big\{\begin{array}{ll}
a_h(\bu_h,\bv_h)+b(p_h,\bv_h)=(\bm f,\bv_h),\\
b(q_h,\bu_h)-{\frac{1}{1+\gamma}}J_h(p_h,q_h)=0
\end{array}
\end{align}
for all $\bv_h\in\bV_h$, $q_h\in Q_h$, where
\[
J_h(q,p)=\sum_{F\in\mcf^{ct,\Gamma}}\sum_{\ell=0}^{k-1}h_F^{2\ell+1}\int_{F}[\partial^\ell_nq][\partial^\ell_np].
\]

\begin{remark}
Compared to other CutFEM discretizations \cite{BurmanHansbo14,cutFEM,St2,GuzmanMaxim18},
we have modified the above finite element method by modifying the terms $\bm{{\rm j}}_h (\bu,\bv)$, and $J_h(q,p).$
In particular, for $\bm{{\rm j}}_h (\bu,\bv)$, and $J_h(q,p)$, instead of summing over all faces $F$ in $\mcf^\Gamma$, we are now summing over all faces $F$ from $\mcf^{ct,\Gamma}.$
Such faces may be completely outside the physical domain $\Omega$.

Likewise, for the term $j_h(\bu,\bv)$, instead of summing over all triangles $T$ from $\mct^\Gamma$, we are now summing over all triangles $K$ from $\mct^{ct,\Gamma}$.  However, since $\bigcup_{K\in\mct^{ct,\Gamma}}K_{\Gamma}=\bigcup_{T\in\mct^{\Gamma}}T_{\Gamma}$, $j_h(\bu,\bv)$ is equivalent to the analogous term  in, e.g., \cite{GuzmanMaxim18}.  Compared to the previous work on unfitted Stokes elements, we also introduce the grad-div stabilization. We show later that it effectively acts only on a narrow boundary strip and allows the control of mass-conservation loss due to ghost penalty stabilization. This mechanism may be of help to boost the accuracy for other unfitted Stokes FE methods.
\end{remark}

\begin{remark}
An accurate computation of integrals over cut elements and $\Gamma$ as they appear in the definition of the bilinear forms is not always feasible. One way to handle this is to introduce a polyhedral approximation of $\Omega$ allowing for standard quadrature rules both on cut and uncut elements.  Such an approximation, however, leads to second order geometric consistency error which is suboptimal for Scott-Vogelius elements. To ensure a geometric error of the same order or higher than the finite element approximation error, we define numerical quadrature rules for cut elements and boundary integrals using the isoparametric approach proposed in \cite{Lehrenfeld16}.
\end{remark}

\subsection{Divergence--free property}
Notice that, in contrast to the domain-fitted Scott-Vogelius FEM,
 the method \eqref{fem} does not produce exactly divergence-free solution
 due to the stability term $J_h(\cdot,\cdot)$ in the discrete continuity equations.
However, we  show that the discrete velocity solution $\bu_h$ is divergence-free on a mesh-dependent interior domain. 

We define the set of all simplices in $\mct^{ct,\Gamma}$
together with those in $\mct^{ct,i}$ that are touching these simplices (cf.~Figure \ref{fig:Meshes}):
\[
\widetilde{\mathcal{T}}_h^{ct,\Gamma} =\{K\in \mct^{ct,e}:\ {\rm meas}_{d-1} (\bar K\cap \bar K')>0\ \exists K'\in \mct^{ct,\Gamma}\}. 
\]
This set's complement is given by an interior set of elements:
\begin{equation}\label{eqn:TildeThCT}
\widetilde{\mathcal{T}}_h^{ct,i} = \mct^{ct,e}\backslash \widetilde{\mathcal{T}}_h^{ct,\Gamma} \subset \mct^{ct,i},
\end{equation}
and we define the domains
\[
\widetilde\Omega_h^{ct,\Gamma} =  {\rm Int}\Big(\bigcup_{K\in \widetilde{\mathcal{T}}_h^{ct,\Gamma}} \bar K\Big),\qquad
\widetilde \Omega_h^{ct,i} = {\rm Int}\Big(\bigcup_{T\in \widetilde{\mathcal{T}}_h^{ct,i}} \bar K\Big).
\]

\begin{lemma}\label{L:divfree}
(Divergence-free property) Suppose that $\bu_h\in \bV_h$ satisfies \eqref{fem}.
Then $\rev{\Div \bu_h} = 0$ on $\widetilde \Omega_h^{ct,i}$.
\end{lemma}
\begin{proof}
We show the result in four steps.
\begin{enumerate}
\item

Fix $K_*\in \widetilde{\mathcal{T}}_h^{ct,i}$, and set
\begin{align*}
q_1 = \left\{
\begin{array}{ll}
1 & \text{on }\widetilde \Omega_h^{ct,\Gamma},\\
-(|\widetilde \Omega_h^{ct,\Gamma}|-|\Omega_h^\Gamma|)/|K_*| & \text{on }K_*,\\
0 & \text{otherwise.}
\end{array}
\right.
\end{align*}
We then have
\begin{align*}
\int_{\Omega_h^i} q_1 = \int_{\widetilde \Omega_h^{ct,\Gamma}\setminus\Omega_h^\Gamma} q_1 + \int_{K_*} q_1
 = (|\widetilde \Omega_h^{ct,\Gamma}|-|\Omega_h^\Gamma|) + |K_*|\Big(-(|\widetilde \Omega_h^{ct,\Gamma}|-|\Omega_h^\Gamma|)/|K_*|\Big) = 0.
\end{align*}
Thus, $q_1\in Q_h$.
We also have $J(p_h,q_1)=0$ because $q_1$ is constant on $\widetilde \Omega^{ct,\Gamma}_h$.
It then follows from \eqref{fem} that
\begin{align*}
- \frac{(|\widetilde \Omega_h^{ct,\Gamma}|-|\Omega_h^\Gamma|)}{|K_*|}\int_{K_*} \rev{\Div \bu_h} 
+ \int_{\Omega\cap \widetilde \Omega_h^{ct,\Gamma}} \rev{\Div \bu_h}
-\int_{\Gamma} \bu_h\cdot \bn= 0,
\end{align*}
and therefore
\begin{align}\label{eqn:DivStep1}
\frac{1}{|K_*|}\int_{K_*} \rev{\Div \bu_h}
%
&=  \frac{1}{|\widetilde \Omega_h^{ct,\Gamma}|-|\Omega_h^\Gamma|} \Big(\int_{\Omega\cap \widetilde \Omega_h^{ct,\Gamma}}
\rev{\Div \bu_h}-\int_{\Omega} \nabla \cdot \bu_h\Big)\\
\nonumber&  = \frac{-1}{|\widetilde \Omega_h^{ct,\Gamma}|-|\Omega_h^\Gamma|} \int_{\widetilde \Omega_h^{ct,i}} \rev{\Div \bu_h}.
%
\end{align}

\item  Fix $K\in \widetilde{\mathcal{T}}_h^{ct,i}\backslash \{K_*\}$, and set
\begin{align*}
q_2 = \left\{
\begin{array}{ll}
1 & \text{on }K,\\
-\frac{|K|}{|K_*|} & \text{on }K_*,\\
0 & \text{otherwise.}
\end{array}
\right.
\end{align*}
Then $q_2\in Q_h$, and by \eqref{fem}, we conclude
\[
\int_K \rev{\Div \bu_h} = \frac{|K|}{|K_*|} \int_{K_*} \rev{\Div \bu_h}.
\]
We then sum this expression over $K\in \widetilde \calT_h^{ct,i}$ to conclude
\begin{align}\label{eqn:DivStep2}
\int_{\widetilde \Omega_h^{ct,i}}  \rev{\Div \bu_h} = \frac{|\widetilde \Omega_h^{ct,i}|}{|K_*|} \int_{K_*} \rev{\Div \bu_h}.
\end{align}

\item We combine \eqref{eqn:DivStep1} and \eqref{eqn:DivStep2} to obtain
\begin{align*}
\int_{\widetilde \Omega_h^{ct,i}} \rev{\Div \bu_h} = -\frac{|\widetilde \Omega_h^{ct,i}|}{|\widetilde \Omega_h^{ct,\Gamma}|-|\Omega_h^\Gamma|} \int_{\widetilde \Omega_h^{ct,i}} \rev{\Div \bu_h},
\end{align*}
which implies
\begin{align*}
\int_{\widetilde \Omega_h^{ct,i}}  \rev{\Div \bu_h} = 0.
\end{align*}
Using \eqref{eqn:DivStep1}, and noting that $K_*\in \widetilde \calT_h^{ct,i}$ was arbitrary, we have
\begin{align*}
\int_K  \rev{\Div \bu_h} = 0\qquad \forall K\in \widetilde \calT_h^{ct,i}.
\end{align*}

\item Fix $K,K_\dagger\in \widetilde{\mathcal{T}}_h^{ct,i}$, and set
\[
q_3 = 
\left\{
\begin{array}{ll}
 \rev{\Div \bu_h}& \text{on }K_\dagger,\\
c & \text{on }K,\\
0 & \text{otherwise,}
\end{array}
\right.
\]
where $c\in \mathbb{R}$ is chosen such that $q_3\in Q_h$.
Then using \eqref{fem},
\begin{align*}
\int_{K_\dagger} | \rev{\Div \bu_h}|^2 = - c \int_{K}  \rev{\Div \bu_h} = 0.
\end{align*}
Thus, $ \rev{\Div \bu_h}= 0$ on $\widetilde \Omega_h^{ct,i}$.

\end{enumerate}
\end{proof}

\section{Stability}\label{sec:Stab}
In this section, we prove inf-sup stability of the finite element method
and derive a priori estimates.  As a first step, 
we state an inf-sup stability result with respect to the finite element
spaces with support on the interior domain $\Omega_h^i$.

To do so, we require
two  mesh-dependent norms
\[
\begin{split}
\|\bu\|_{\bV_{0,h}}^2 &= |\bu|^2_{H^1(\Omega)}+\eta j_h(\bu,\bu)+\bm{{\rm j}}_h(\bu,\bu),\\
\|\bu\|_{\bV_h}^2 &=\|\bu\|_{\bV_{0,h}}^2+ \gamma\|\Div\bu\|^2_{L^2(\Omega)},
\end{split}
\]
as well as the associated dual norm
\[
\|{\bm f}\|_{\bV_h'} = \sup_{\bv\in \bV_h\backslash \{0\}} \frac{({\bm f},\bv)}{\|\bv\|_{\bV_h}}.
\]
\begin{remark}\label{rem:NormComment}
By definition, there holds $\|\bz\|_{\bV_{h}}\le (1+\gamma)^{1/2}\|\bz\|_{\bV_{0,h}}$
for all $\bz\in \bV_h$.
\end{remark}

We first need an inf-sup stability estimate for the SV element in the interior domain $\Omega_h^i$. The estimate is formulated below in Theorem~\ref{Thm1}. Note that $\Omega_h^i$ is mesh dependent and \eqref{inf-sup} does not  follow easily from an  ``inf-sup-stability'' of any finite element pair (see discussion in \cite{GuzmanMaxim18}). The property \eqref{inf-sup} was only \emph{assumed} to hold in earlier publications, e.g. \cite{BurmanHansbo14,St2}, on unfitted FEMs for the Stokes problem, and has been proved for $\bpol_2-\pol_1$ element in \cite{St3} and for several other FE pairs in \cite{GuzmanMaxim18}. The latter paper does not cover any divergence-free elements.   

\begin{theorem}\label{Thm1}
There exists a constant $\theta>0$ and a constant $h_0>0$ 
such that  we have the following result for $h\leq h_0$
\begin{align}\label{inf-sup}
\theta \|q\|_{L^2(\Omega_h^i)}\leq \sup_{\bv\in \bV_h^i\backslash \{0\}}\frac{\int_{\Omega_h^i}  (\Div \bv)q }{\|\bv\|_{H^1(\Omega_h^i)}}    
\qquad \forall q\in Q_h^i.
\end{align}
The constant $\theta>0$ is independent of $h$.
\end{theorem}
The proof of this theorem in the case $k=d=2$ can be found in \cite[Section 4.2]{LNO}.
The general case $k\ge d$ ($d\in \{2,3\}$) follows verbatim using the $\bpol_d-\pol_0$
stability result in \cite{GuzmanMaxim18}: A local inf-sup stability result for $q\in\pol_{k-1}(K)$ with $\int_K q=0$ on the CT split of each $K\in\calT_h^i$  (Theorem~3.1 in~\cite{Guz2}) is applied together with the global inf-sup stability of  $\bpol_d-\pol_0$ element for the macro-triangulation $\calT_h^i$. These two results are combined by standard arguments, see e.g. Lemma~4.7 in \cite{LNO} or Proposition~6.1 in \cite{Guz2}, to yield \eqref{inf-sup}. We skip including further details.


\begin{corollary}\label{corollary1}
The following stability is satisfied
\begin{align}\label{c1}
   \theta_*\|q\|_{L^2(\Omega)}\le \mathop{\sup_{\bv\in \bV_h\backslash \{0\}}}_{{\rm supp}(\bv)\subset\Omega_h^i}\frac{b(\bv,q)}{\|\bv\|_{\bV_{0,h}}}+J_h^{1/2}(q,q)\qquad \forall\ q\in Q_h,
 \end{align}
where $\theta_* > 0$ is independent of $h$ and the position of $\Gamma$ in the mesh.
\end{corollary}
\begin{proof}
Fix some $q\in Q_h$. 
By \cite[Lemma 5.1]{MassingLarsonLoggRognes},
 for each pair of triangles $K_1$ and $K_2$ in $\mct^{ct}$ with $\p K_1\cap \p K_2=F\in \mcf^{ct,e}$
 we have
\begin{align*}
 \|q\|^2_{L^2(K_1)}\le C\Big(\|q\|^2_{L^2(K_2)}+\sum_{\ell=0}^{k-1}h_F^{2\ell+1}\int_{F}[\partial^\ell_nq]^2 \Big).
\end{align*}
Iterating this estimate, we conclude
\begin{equation}\label{aux473}
\|q\|_{L^2(\Omega)}^2 \le \|q\|_{L^2(\Omega_h^e)}^2 \le C\big(\|q\|_{L^2(\Omega_h^i)}^2 + J_h(q,q)\big).
\end{equation}
Combining this estimate with Theorem \ref{Thm1},
we conclude that there exists $\bv\in \bV_h$ with ${\rm supp}(\bv)\subset \Omega_h^i$
such that
\begin{align}\label{c1s1}
 \|q\|^2_{L^2(\Omega)}\le C\big(\|q\|^2_{L^2(\Omega_h^i)}+J_h(q,q)\big)\le C \theta^{-1} \Big(\frac{\int_{\Omega_h^i} (\Div \bv)q}{\|\bv\|_{H^1(\Omega_h^i)}}+J_h(q,q)\Big).
\end{align}
Because $\bv=0$ on $\Omega_h^\Gamma$, we have $j_h(\bv,\bv)=0$ and by an inverse estimate,
\begin{align*}
    \bm{{\rm j}}_h(\bv,\bv)
    &=\mathop{\sum_{F\in \mcf^{ct,\Gamma}}}_{F\subset \partial\Omega_h^i}\sum_{\ell=1}^{k}h_F^{2\ell-1}\int_{F}[\partial^\ell_n\bv]^2\\
    &\le C \sum_{K\in \widetilde{\mathcal{T}}_h^{ct,\Gamma}\cap \mct^{ct,i}}\|\nabla \bv\|^2_{L^2(K)}\le C\|\bv\|^2_{H^1(\Omega_h^i)}.
\end{align*}
Thus we have $\|\bv\|_{\bV_{0,h}}\le C  \|\bv\|_{H^1(\Omega_h^i)}. $ Combining this with \eqref{c1s1}, we have \eqref{c1}.
\end{proof}

\subsection{A priori estimates for the finite element method}
In this section we derive a priori estimates of the finite element
method, thus showing that the discrete problem \eqref{fem} is well--posed.
The techniques to show these results are rather standard, but 
we show the proofs here for completeness.

\begin{lemma}\label{lem:aCCont}
There exists constants $C_a,C_0>0$ such that
\begin{alignat*}{2}
a_h(\bu,\bv)&\le C_a\|\bu\|_{\bV_h}\|\bv\|_{\bV_h}\quad &&\forall \bu,\bv\in \bV_h+\bH^{k+1}(\Omega_h^e),\\
 C_0\|\bv\|^2_{\bV_h}&\le a_h(\bv,\bv)\ &&\forall \bv\in \bV_h .
 \end{alignat*}
\end{lemma}
\begin{proof}
The proof of this result can be found in, e.g., \cite{MassingLarsonLoggRognes,GuzmanMaxim18,BurmanHansbo12}.
\end{proof}

\begin{theorem}\label{thm:WP}
Suppose that $(\bu_h,p_h)\in\bV_h\times Q_h$ satisfies \eqref{fem}. Then
\begin{align}\label{apriori}
\| p_h\|_{L^2(\Omega)}
\leq C(1+\gamma)^{\frac12}\|{\bm f}\|_{\bV_{h}'},\quad
\|\bu_h\|_{\bV_h}\leq C\|{\bm f}\|_{\bV_h'},
\end{align}
for some $C>0$ independent of $\gamma$, $h$, and the position of $\Gamma$ in the mesh.
Consequently, \eqref{fem} has a unique solution.
\end{theorem}
\begin{proof}
We set $\bv_h=\bu_h$ in the first equation in \eqref{fem}, and $q_h=p_h$ in the second equation of \eqref{fem}
and subtract the resulting expressions:
\[
a_h(\bu_h,\bu_h)+{\frac{1}{1+\gamma}}J_h(p_h,p_h)=({\bm f},\bu_h).
\]

By the coercivity of $a_h(.,.)$ stated in Lemma \ref{lem:aCCont} and the Cauchy-Schwarz inequality, we have 
\begin{align*}
C_0\|\bu_h\|^2_{\bV_h}+{\frac{1}{1+\gamma}}J_h(p_h,p_h)&\leq({\bm f},\bu_h)
\leq \|\bm f\|_{\bV_h'}\|\bu_h\|_{\bV_h},
\end{align*}
and so
\begin{align}\label{bound1}
\frac{C_0}{2} \|\bu_h\|^2_{\bV_h} + {\frac{1}{1+\gamma}} J_h (p_h,p_h) \le \frac{1}{2C_0} \|{\bm f}\|_{\bV_h'}^2.
\end{align}

By the inf-sup stability estimate \eqref{c1} and Remark \ref{rem:NormComment}, 
there exists $\bz\in\bV_h$ with $(1+\gamma)^{-\frac12}\|\bz\|_{\bV_{h}}\le \|\bz\|_{\bV_{0,h}}=\|p_h\|_{L^2(\Omega)}$ and
\begin{align*}
   \theta_* \|p_h\|_{L^2(\Omega)}^2
    &\leq b(\bz, p_h)+J_h^{1/2}(p_h,p_h)\|p_h\|_{L^2(\Omega)}\\
& =     
    (\bm f,\bz)-a_h(\bz,\bu_h)+J_h^{1/2}(p_h,p_h)\|p_h\|_{L^2(\Omega)}.
\end{align*}
By Lemma \ref{lem:aCCont} and the Cauchy-Schwarz inequality, we have
\begin{align*}
    \theta_*\|p_h\|_{L^2(\Omega)}^2
    &\le \big(\|{\bm f}\|_{\bV_{h}'} \|\bz\|_{\bV_{h}} + C_a \|\bz\|_{\bV_{h}} \|\bu_h\|_{\bV_{h}}\big)
    +J^{1/2}_h(p_h,p_h)\|p_h\|_{L^2(\Omega)}\\
    & \le  \big((1+\gamma)^{\frac12}(\|{\bm f}\|_{\bV_h'}  + C_a\|\bu_h\|_{\bV_{h}} )
    +J^{1/2}_h(p_h,p_h)\big)\|p_h\|_{L^2(\Omega)}.
\end{align*}
Dividing by $\|p_h\|_{L^2(\Omega)}$ and using \eqref{bound1},
we conclude
\begin{align*}
    \theta_*^2\|p_h\|_{L^2(\Omega)}^2
    & \le  3\big((1+\gamma)(\|{\bm f}\|^2_{\bV_{h}'}  + C^2_a \|\bu_h\|^2_{\bV_{h}})
    +J_h(p_h,p_h)\big)\\
    &\le 3(1+\gamma)\Big(1+ C_a^2 C_0^{-2} + \frac{1}{2C_0} \Big)\|{\bm f}\|_{\bV_{h}'}^2.
\end{align*}
This estimate and \eqref{bound1} yields the desired result \eqref{apriori}.
\end{proof}

\section{Convergence Analysis}\label{sec:Conv}
In this section we assume that the solution to the Stokes problem \eqref{eqn:Stokes1} is sufficiently
smooth, i.e., $\bu\in \bH^{k+2}(\Omega)$, $p\in H^{k+1}(\Omega)$, where we recall $k$ is the polynomial degree in the definition of finite element spaces. 
Without loss of generality, we assume that ${\rm dist}(\p S,\p\Omega) = O(1)$.

Because $\partial\Omega$ is Lipschitz
there exists an extension of $p$, which we also denote by $p$, such that $p\in H^{k+1}(S)$ and (cf.~\cite{Stein})
\begin{subequations}
\label{Ext}
\begin{equation}
\|p\|_{H^\ell(S)}\le C\|p\|_{H^\ell(\Omega)}\quad\text{ for $\ell=0,1,\ldots,k+1$}.
\end{equation}
An analogous extension of $\bu$ is done in the following manner.
First, write the velocity in terms of potential field
$\bu = \text{\bf curl}\bpsi$, where we agree to understand $\text{\bf curl}\bpsi$ for any space dimension $d$ as the exterior derivative of $(d-2)$-differential form $\bpsi$.
For $\bu\in \bH^{k+2}(\Omega)$, the form $\bpsi$ satisfies $\bpsi\in \bH^{k+3}(\Omega)$
and $\|\bpsi\|_{H^{\ell+1}(\Omega)}\le C \|\bu\|_{H^{\ell}(\Omega)}$ for $\ell=0,1,\ldots k+2$ \cite{GR,CostabelMcIntosh2010}.
We extend $\bpsi$ to $S$ in a way such that $\|\bpsi\|_{H^\ell(S)}\le C \|\bpsi\|_{H^\ell(\Omega)}$ for $\ell=0,1,\ldots,k+3$,
and let  $\omega$ be a smooth cut-off function with compact support in $S$ and $\omega\equiv 1$ in $\Omega$.
We then define the velocity extension as $\bu = \text{\bf curl} (\omega \bpsi)$, so that
$\bu$ is divergence--free, vanishes on $\p S$, and 
	\begin{equation}
		\|\bu\|_{H^{\ell}(S)}\le C 	\|\bu\|_{H^{\ell}(\Omega)}\quad \text{ for $\ell=0,1,\ldots,k+2$}.
\end{equation}
\end{subequations}

\begin{remark}[Consistency]
Standard arguments show that the method \eqref{fem} is consistent.
In particular, there holds
\begin{align}\label{eqn:Consistent}
\Big\{\begin{array}{ll}
a_h(\bu,\bv_h)+b(p,\bv_h)=(\bm f,\bv_h),\\
b(q_h,\bu)-{\frac{1}{1+\gamma}}J_h(p,q_h)=0
\end{array}
\end{align}
for all $\bv_h\in\bV_h$, $q_h\in Q_h$.
\end{remark}

The following lemma is a direct application of
\cite[Lemma~4.10]{ElliotRan}.
\begin{lemma}\label{lem:ElliottRan}
For $T\in \calT_h^e$, define 
$\displaystyle\omega_T = \mathop{\cup_{T'\in \calT_h^e}}_{\bar T\cap \bar T'\neq \emptyset} \bar T'$ to be the patch
of neighboring elements of $T$.  We further define the $O(h)$ strip around $\Gamma$:
\[
\omega_\Gamma = \bigcup_{T\in \calT_h^\Gamma}\omega_T.
\]
Then there holds
\[
\|v\|_{L^2(\omega_\Gamma)}\le C h^{\frac12} \|v\|_{H^1(S)}\qquad \forall v\in H^1(S).
\]
\end{lemma}

We also require a trace inequality suitable
for the CutFEM discretization  (see, e.g., \cite{hansbo2002unfitted,GuzmanMaxim18}).
\begin{lemma}\label{lem:TraceonGamma}
For every  $K\in \mathcal{T}_h^{ct,\Gamma}$ it holds
\begin{equation}\label{trace}
\|v\|_{L^2(K_\Gamma)}\le C(h^{-\frac12}_K\|v\|_{L^2(K)}+h^{\frac12}_K\|\nabla v\|_{L^2(K)})\qquad\forall v\in H^1(K),
\end{equation}
with a constant $C$ independent of $v$, $T$,  how $\Gamma$ intersects $T$, and $h<h_1$ for some  fixed $h_1>0$.
\end{lemma}

Consider  the finite element subspace of pointwise divergence-free functions:
\[
\bZ_h=\{\bw_h\in\bV_h\,:\,\Div \bw_h=0~\text{in}~\Omega_h^e\}.
\]
This subspace enjoys full approximation properties in the sense of the following lemma.  
\begin{lemma}\label{lem:LocalKernelApprox} 
For $\bu$, the divergence--free extension of the solution to \eqref{eqn:Stokes}, it holds
	\begin{equation}\label{ApproxProp1}
		\inf_{\bw_h\in\bZ_h}\|\bu-\bw_h\|_{H^1(T)}\le C h_T^k|\bu|_{H^{k+1}(\omega_T)}\qquad \forall T\in \calT_h^e.
	\end{equation}
	  Consequently, if $\bu\in \bH^{k+2}(\Omega)$,
\begin{align*}
\inf_{\bw_h\in \bZ_h} \|\bu-\bw_h\|_{\bV_h} \le C \big(h^k\|\bu\|_{H^{k+1}(\Omega)} + \eta^{\frac12} h^{k+\frac12} \|\bu\|_{H^{k+2}(\Omega)}\big).
\end{align*}
\end{lemma}
\begin{proof}
The proof of \eqref{ApproxProp1} is given in the appendix.

To bound $\|\bu-\bw_h\|_{\bV_h}$ we use the approximation results from
\eqref{ApproxProp1} and note that $\Div(\bu-\bw_h)=0$ in $\Omega_h^e$. 
The penalty part of $ \|\bu-\bw_h\|_{\bV_h}$ is estimated using Lemma \ref{lem:TraceonGamma} as follows:
\begin{equation}\label{aux674}
\begin{split}
	j_h(\bu-\bw_h,\bu-\bw_h)& =   \sum_{K\in \calT_h^{ct,\Gamma}} h_K^{-1} \|\bu-\bw_h\|^2_{L^2(K_\Gamma)}\\ 
	&\le C \sum_{K\in \calT_h^{ct,\Gamma}} \big( h_K^{-2} \|\bu-\bw_h\|_{L^2(K)}^2 + \|\nab (\bu-\bw_h)\|_{L^2(K)}^2\big)\\
	&\le C h^{2k}\|\bu\|_{H^{k+1}(\omega_\Gamma)}^2
	\le C h^{2k+1}\|\bu\|_{H^{k+2}(\Omega)}^2,
\end{split}
\end{equation}
where for the last inequality we used \eqref{Ext} and Lemma \ref{lem:ElliottRan}.
This yields the bound $\|\bu-\bw_h\|_{\bV_h}\le C\big(h^{k}\|\bu\|_{H^{k+1}(\Omega)}
+\eta^{\frac12} h^{k+\frac12}\|\bu\|_{H^{k+2}(\Omega)}\big)$.
\end{proof}

\begin{theorem}\label{thm:MainError} The following error estimate holds
\begin{multline}\label{VelErr}
	\|\bu-\bu_h\|_{\bV_h} +  (1+\gamma)^{-\frac12}\|p- p_h\|_{L^2(\Omega)} 
	\le C\Bigg(h^{k} \|\bu\|_{H^{k+1}(\Omega)}
	+(1+\gamma^{\frac12}+ \eta^{\frac12}) h^{k+\frac12}\|\bu\|_{H^{k+2}(\Omega)}\\
	+ (\eta^{-\frac12}+{(1+\gamma)^{-\frac12}}) h^{k+\frac12}\|p\|_{H^{k+1}(\Omega)}+ 
	(1+\gamma)^{-\frac12} h^{k}\|p\|_{H^{k}(\Omega)}\Bigg).
\end{multline}
\end{theorem}
\begin{proof}	
To show the error bounds, we start with a standard argument.  
Let $\bw_h\in \bZ_h$ be a function in the discrete kernel satisfying estimate \eqref{ApproxProp1}.
Setting $\be_I=\bu_h-\bw_h\in\bV_h$ we have, thanks to the coercivity result in Lemma~\ref{lem:aCCont}:
\begin{equation}\label{aux609}
		C_0\|\be_I\|^2_{\bV_h}\le a_h(\be_I,\be_I). 
\end{equation}

Denote by $\widehat p_h\in Q_h$ the $L^2$-projection of $p$ onto $Q_h$,
and set $q_I=p_h-\widehat p_h$.   It follows from \eqref{c1} that  there exists  ${\bv\in \bV_h}$ with ${\rm supp}(\bv)\subset \Omega_h^i$ 
such that 
\begin{equation}\label{aux638}
		C_1\|q_I\|_{L^2(\Omega)}^2\le b(\bv,q_I)+C_1^{-1}J_h(q_I,q_I),\quad \text{with}~(1+\gamma)^{-\frac12} \|\bv\|_{\bV_h}\le \|\bv\|_{\bV_{0,h}}=\|q_I\|_{L^2(\Omega)},
\end{equation}
where $C_1 = \frac{\theta_*}{2}$, and $\theta_*$ is the inf-sup constant given in Corollary \ref{corollary1}.

From \eqref{aux609}, \eqref{aux638}, \eqref{fem}, and the consistency identity \eqref{eqn:Consistent}, we conclude that for any $\alpha\ge0$ it holds
\begin{align}
\label{aux654}
&C_0\|\be_I\|^2_{\bV_h}+C_1\alpha\|q_I\|_{L^2(\Omega)}^2 +\big((1+\gamma)^{-1}-\alpha C_1^{-1}\big)J_h(q_I,q_I)\\
&\nonumber\qquad \le a_h(\be_I,\be_I) + b(\alpha \bv,q_I)+ (1+\gamma)^{-1} J_h(q_I,q_I)\\
&\nonumber\qquad= a_h(\be_I,\be_I+\alpha \bv) + b(\be_I+\alpha \bv,q_I)  - b(\be_I,q_I) + (1+\gamma)^{-1} J_h(q_I,q_I) -a_h(\be_I,\alpha \bv)\\
&\nonumber\qquad= a_h(\bu-\bw_h,\be_I+\alpha \bv) + b(\be_I+\alpha \bv,p-\widehat p_h)  - b(\bu-\bw_h,q_I) \\
&\nonumber\qquad\qquad+ (1+\gamma)^{-1} J_h(p-\widehat p_h,q_I) -a_h(\be_I,\alpha \bv)\\
 &\nonumber\qquad=:I_1+I_2+I_3+I_4+I_5.
\end{align}	
We now estimate the right-hand side of \eqref{aux654} term-by-term.

Using the continuity  result in Lemma~\ref{lem:aCCont} and the approximation results in Lemma \ref{lem:LocalKernelApprox}, we bound  
\begin{align}\label{aux609ABC}
I_1
 &\le C_a\|\be_I+\alpha\bv\|_{\bV_h}\|\bu-\bw_h\|_{\bV_h}\\
 &\nonumber\le C\|\be_I+\alpha\bv\|_{\bV_h}\left(h^{k}\|\bu\|_{H^{k+1}(\Omega)}+\eta^{\frac12} h^{k+\frac12}\|\bu\|_{H^{k+2}(\Omega)}\right)\\
 &\nonumber\le C \big(\|\be_I\|_{\bV_h} + \alpha (1+\gamma)^{\frac12} \|q_I\|_{L^2(\Omega)}\big)\left(h^{k}\|\bu\|_{H^{k+1}(\Omega)}+\eta^{\frac12} h^{k+\frac12}\|\bu\|_{H^{k+2}(\Omega)}\right),
\end{align}
where we used \eqref{aux638} in the last inequality.



We now estimate the second term in the right-hand side of \eqref{aux654} in two steps. 
First, using approximation properties of the $L^2$-projection, we get
\begin{equation}\label{aux619}
(\widehat p_h-p,\Div\be_I )\le (1+\gamma)^{-\frac12}{\|\widehat p_h-p\|_{L^2(\Omega)}}  \|\be_I\|_{\bV_h}\le (1+\gamma)^{-\frac12}h^{k}\|p\|_{H^{k}(\Omega)}  \|\be_I\|_{\bV_h}.
\end{equation}
Likewise,
\begin{equation}\label{aux619b}
	\begin{split}
		(\widehat p_h-p,\alpha \Div \bv )\le C \alpha
		h^{k}\|p\|_{H^{k}(\Omega)}  \|q_I\|_{L^2(\Omega)}. 
	\end{split}
\end{equation}

We apply the trace inequality \eqref{trace} and standard approximation properties
of the $L^2$-projection to estimate the boundary integral in $b(\be_I+\alpha\bv,p-\widehat p_h)$, noting that $\bv=0$ on $\Gamma$: 
\begin{equation}\label{aux635}
\begin{split}
\int_{\Gamma}(\widehat p_h-p)(\be_I+\alpha\bv)\cdot \bn
&\le \sum_{K\in \calT_h^{ct,\Gamma}} \|\widehat p_h-p\|_{L^2(K_\Gamma)} \|\be_I\|_{L^2(K_\Gamma)}\\
&\le \Big(\sum_{K\in \calT_h^{ct,\Gamma}}  \eta^{-1} h_K \|\widehat p_h-p\|^2_{L^2(K_\Gamma)}\Big)^{1/2} \Big(\sum_{K\in \calT_h^{ct,\Gamma}} \eta h_K^{-1} \|\be_I\|_{L^2(K_\Gamma)}\Big)^{1/2}\\
&\le C \eta^{-\frac12} h^{k}\|p\|_{H^k(\omega_\Gamma)} \|\be_I\|_{\bV_h} \le C  \eta^{-\frac12} h^{k+\frac12}\|p\|_{H^{k+1}(\Omega)} \|\be_I\|_{\bV_h},
\end{split}
\end{equation}
where we used Lemma \ref{lem:ElliottRan} in the last inequality.
Summing \eqref{aux619}--\eqref{aux635} we obtain
\begin{equation}\label{aux715}
I_2\le C\Big(
{\alpha h^{k}\|p\|_{H^{k}(\Omega)} } \|q_I\|_{L^2(\Omega)}+ 
{\left((1+\gamma)^{-\frac12}  h^k \|p\|_{H^k(\Omega)}
+\eta^{-\frac12} h^{k+\frac12}\|p\|_{H^{k+1}(\Omega)} \right) \|\be_I\|_{\bV_h}}\Big).
\end{equation}


To estimate $I_3$, we first note that, due to  \eqref{trace}, finite element inverse inequalities, and \eqref{aux473},
there holds
\begin{align*}
 \sum_{K\in \calT_h^{ct,\Gamma}} h_K\|q_I\|_{L^2(K_\Gamma)}^2 \le C  \sum_{K\in \calT_h^{ct,\Gamma}}\|q_I\|_{L^2(K)}^2\le C \big(\|q\|_{L^2(\Omega)}^2 + J_h(q_I,q_I)\big).
\end{align*}
Therefore, thanks to $\Div (\bu-\bw_h)=0$ and the estimate  \eqref{aux674}, we have
\begin{align*}
I_3
&\le \Big( \sum_{K\in \calT_h^{ct,\Gamma}} h_K\|q_I\|_{L^2(K_\Gamma)}^2\Big)^{1/2}\Big( \sum_{K\in \calT_h^{ct,\Gamma}}  h_K^{-1}\|\bu-\bw_h\|_{L^2(K_\Gamma)}^2\Big)^{1/2}\\
& \le C (\|q_I\|^2_{L^2(\Omega)}+J_h(q_I,q_I))^\frac12 h^{k+\frac12}\|\bu\|_{H^{k+2}(\Omega)}.
\end{align*}

We proceed with estimating  terms in the right-hand side of \eqref{aux654}. For the fourth term we get, using
the trace inequality \eqref{trace}, approximation properties of the $L^2$-projection, and Lemma \ref{lem:ElliottRan},
\begin{equation}\label{aux731}
I_4 \le  (1+\gamma)^{-1}	J_h^\frac12(p-\widehat p_h,p-\widehat p_h)J^\frac12_h(q_I,q_I)\le C {h^{k+\frac12}} (1+\gamma)^{-1} \|p\|_{H^{k+1}(\Omega)}J^\frac12_h(q_I,q_I).
\end{equation}
	
For the last term in \eqref{aux654} we have ({using \eqref{aux638})}
\begin{equation}\label{aux736}
I_5
\le \frac{C_0}{4}\|\be_I\|_{\bV_h}^2 + \frac{C_a^2 \alpha^2}{C_0}\|\bv\|_{\bV_h}^2
\le \frac{C_0}{4}\|\be_I\|_{\bV_h}^2 + \frac{C_a^2 \alpha^2(1+\gamma)}{C_0}\|q_I\|^2_{L^2(\Omega)}.
\end{equation}

We  apply the estimates \eqref{aux609ABC}--\eqref{aux736}
to \eqref{aux654} getting
\begin{align*}
&C_0\|\be_I\|^2_{\bV_h}+C_1\alpha\|q_I\|_{L^2(\Omega)}^2 +((1+\gamma)^{-1}-\alpha C_1^{-1})J_h(q_I,q_I)\\
&\le C\Bigg( \big(\|\be_I\|_{\bV_h}  + \alpha(1+\gamma)^{\frac12}\|q_I\|_{L^2(\Omega)}\big) \big(h^k\|\bu\|_{H^{k+1}(\Omega)} + \eta^{\frac12} h^{k+\frac12}\|\bu\|_{H^{k+2}(\Omega)}\big)\\
&\qquad+ \alpha h^{k} \|q_I\|_{L^2(\Omega)} \|p\|_{H^{k}(\Omega)}
 +  \Big((1+\gamma)^{-\frac12} h^k \|p\|_{H^k(\Omega)} + \eta^{-\frac12} h^{k+\frac12} \|p\|_{H^{k+1}(\Omega)}\Big)\|\be_I\|_{\bV_h}\\
&\qquad\qquad+ (\|q_I\|_{L^2(\Omega)} + J^\frac12_h(q_I,q_I))h^{k+\frac12}(\|\bu\|_{H^{k+2}(\Omega)}+(1+\gamma)^{-1} \|p\|_{H^{k+1}(\Omega)})\Bigg)\\
&\qquad\qquad\qquad +\frac{C_0}{4}\|\be_I\|_{\bV_h}^2 + \frac{C_a^2\alpha^2(1+\gamma)}{C_0}\|q_I\|^2_{L^2(\Omega)}.
\end{align*}	

We apply the Cauchy-Schwarz inequality several times and rearrange terms to obtain
\begin{align*}
&C_0\|\be_I\|^2_{\bV_h}+\left(C_1\alpha-C {\alpha^2 (1+\gamma)}\right)\|q_I\|_{L^2(\Omega)}^2 +((1+\gamma)^{-1}-\alpha C_1^{-1})J_h(q_I,q_I)\\
&\qquad\le C\Bigg(  
\Big(h^{2k} \|\bu\|_{H^{k+1}(\Omega)}^2 + \eta h^{2k+1}\|\bu\|_{H^{k+2}(\Omega)}^2\Big)
+ (\alpha+1)(1+\gamma)^{-1}h^{2k}\|p\|_{H^k(\Omega)}^2\\
&\qquad\qquad
 + \eta^{-1} h^{2k+1}\|p\|_{H^{k+1}(\Omega)}^2
+ \alpha^{-1}h^{2k+1}\big(\|\bu\|_{H^{k+2}(\Omega)}^2 + {(1+\gamma)^{-2}}\|p\|_{H^{k+1}(\Omega)}^2\big)\Bigg).
\end{align*}	

We now take $\alpha = \widetilde C (1+\gamma)^{-1}$, with $\widetilde C>0$ sufficiently small to obtain
\begin{align*}
&C_0\|\be_I\|^2_{\bV_h}+C(1+\gamma)^{-1}\big(\|q_I\|_{L^2(\Omega)}^2 +J_h(q_I,q_I)\big)\\
%
&\qquad \le C \Bigg( h^{2k} \|\bu\|_{H^{k+1}(\Omega)}^2+ (1+\gamma)^{-1} h^{2k} \|p\|_{H^k(\Omega)}^2 + \big(1+ \eta +\gamma\big)h^{2k+1}\|\bu\|_{H^{k+2}(\Omega)}^2\\
&\qquad\qquad + (\eta^{-1}+{(1+\gamma)^{-1}}) h^{2k+1} \|p\|_{H^{k+1}(\Omega)}^2\Bigg).
\end{align*}	
Finally, we apply the triangle inequality, the divergence-free property of $\bu$ and $\bw_h$,
and approximation properties \eqref{ApproxProp1} to obtain the error estimate  \eqref{VelErr}.

\end{proof}

\begin{remark}
	The pressure dependence in velocity error \eqref{VelErr}
	arises from the violation of mass conservation in the boundary strip and the penalty treatment of the boundary condition.  The violation of the divergence-free constraint in a boundary strip can be partially mitigated by taking grad-div parameter $\gamma=O(h^{-1})$
	and Nitsche parameter $\eta = O(h^{-1})$, which seem to be the optimal choice with respect to the error analysis in the energy norm. This can be contrasted to $\gamma=O(1)$ for the Taylor-Hood element.  
\end{remark}

\section{Extensions to Powell-Sabin Splits}\label{sec:PS}
In this section, we extend
the method and analysis in the previous sections
to the Scott-Vogelius finite element pair
on two-dimensional Powell-Sabin splits.
For brevity, we concentrate
on the lowest-order pair
which has recently been shown to be inf-sup stable in a non-cutFEM setting
in \cite{FabEtAl22}.

As in the previous sections, we let
$\calT_h$ be a simplicial mesh of $S$ (with $\bar \Omega \subset S$),
and let $\calT_h^i$ and $\Omega_h^i$ be the set of interior simplices and interior domain, respectively,
defined by \eqref{eqn:mctiDef}.  Let $\calT_h^\Gamma$
be the sets of simplicies that cut through the interface,  
$\calT_h^{e} = \calT_h^i\cup \calT_h^\Gamma$,
and $\calF_h^\Gamma$ to be the set of edges in $\calT_h^\Gamma$
that do not lie on $\p\Omega_h^e$.

For each $T\in \calT_h$,
we denote by $z_T$ the incenter of $T$.
The Powell-Sabin refinement of $\calT_h^\circ$ ($\circ\in \{i,e,\Gamma\}$),
denoted by  $\calT_h^{ps,\circ}$,
is constructed in three steps as follows:  (1) similar to the Clough-Tocher refinement, we connect
the incenter $z_T$ of each $T\in \calT_h^\circ$ with its vertices; (2) for each interior edge $e$ of $\calT_h^\circ$, 
with $e = \p T_1\cap \p T_2$, we add a vertex
(on $e$) by connecting the incenters $z_{T_1}$ and $z_{T_2}$ by a straight line; 
(3) for each boundary edge $e$ with $e\subset \p T$, we add a vertex by connecting
the incenter $z_T$ with the edge midpoint of $e$.

Thus, we see that the Powell-Sabin refinement splits
each triangle $T\in \calT_h^\circ$ into six sub-triangles.
Further, this refinement  produces many singular vertices, i.e.,
vertices that fall on exactly two straight lines in the mesh.  These vertices
are exactly those produced in steps (2) and (3) of the above procedure.

Let $\calS_h^{\circ,I}$
and $\calS_h^{\circ,B}$ 
be the sets of interior and boundary singular vertices in $\calT_h^{ps,\circ}$,
and set $\calS_h^\circ = \calS_h^{\circ,I}\cup \calS_h^{\circ,B}$ the set
of singular of vertices of $\calT_h^{ps,\circ}$.
For $z\in \calS_h^{\circ,I}$, we denote by $\calT_z\subset \calT_h^{ps,\circ}$
the set of four triangles
that have $z$ as a vertex. We write $\calT_z = \{K_z^{(1)},K_z^{(2)},K_z^{(3)},K_z^{(4)}\}$, labeled
such that $K_z^{(j)}$ and $K_z^{(j+1)}$ have a common edge; see Figure \ref{fig:PSSing}.
For a boundary singular vertex $z\in \calS_h^{\circ,B}$
we let $\calT_z = \{K_z^{(1)},K_z^{(2)}\}\subset \calT_h^{ps,\circ}$,
the set of two triangles that have $z$ as a vertex.

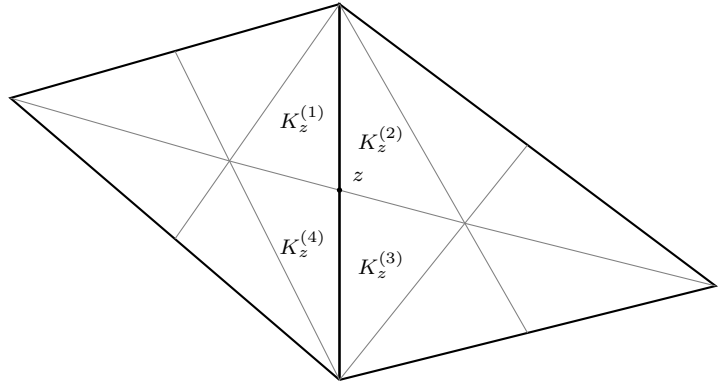
\begin{figure}[h]
\centering
\begin{tikzpicture}[scale=1.25]
\draw[-,thick](0,0)--(0,4)--(-3.5,3)--(0,0);
\draw[-,thick](0,0)--(0,4)--(4,1)--(0,0);

\node (z1) at (-1.17,2.33) {}; 
\node (z2) at (1.333333,1.666667) {}; 
\node (z) at (0,2.066667) {}; 

\draw[-,gray](0,0)-- (-1.17,2.33);
\draw[-,gray](0,4)-- (-1.17,2.33);
\draw[-,gray](-3.5,3)-- (-1.17,2.33);

\draw[-,gray](0,0)-- (1.333333,1.666667) ;
\draw[-,gray](0,4)-- (1.333333,1.666667) ;
\draw[-,gray](4,1)-- (1.333333,1.666667) ;

\draw[-,gray](-1.17,2.33)--(1.333333,1.666667);

\draw[-,gray](1.333333,1.666667)--(2,2.5);
\draw[-,gray](1.333333,1.666667)--(2,0.5);
\draw[-,gray](-1.17,2.33)--(-1.75,1.5);
\draw[-,gray](-1.17,2.33)--(-1.75,3.5);

\node[inner sep = 0pt,minimum size=2pt,fill=black!100,circle] (n2) at (0,2.019973483)  {};
\draw (0.2,2.166667) node {\footnotesize $z$};

\draw (-0.390000000000000,2.78332449433333) node {\footnotesize $K_z^{(1)}$};
\draw (0.444444333333333,2.56221349433333) node {\footnotesize $K_z^{(2)}$};
\draw (0.444444333333333,1.22888016100000) node {\footnotesize $K_z^{(3)}$};
\draw (-0.390000000000000,1.44999116100000) node {\footnotesize $K_z^{(4)}$};

\end{tikzpicture}
\caption{\label{fig:PSSing}Example of an interior singular vertex in the Powell-Sabin refinement.}
\end{figure}

The following result states a weak continuity property
of the divergence acting on piecewise smooth functions at singular
vertices.  The proof can be found in, e.g., \cite{SV1}.
\begin{proposition}\label{prop:DivOnPS}
For a piecewise smooth function
$q$ with respect to $\calT_h^{ps,\circ}\ (\circ\in \{i,e\})$, define
\[
\theta_z(q) = 
\left\{
\begin{array}{ll}
(q_1-q_2+q_3-q_4)(z) & \text{if $z\in \calS_h^{\circ,I}$},\\
(q_1-q_2)(z) & \text{if $z\in \calS_h^{\circ,B}$},
\end{array}
\right.
\]
where $q_j = q|_{K_z^{(j)}}$.
Then there holds for all piecewise smooth $\bv\in \bH^1_0(\Omega_h^\circ)$,
\[
\theta_z(\Div \bv) = 0\qquad \forall z\in \calS_h^\circ.
\]
Moreover, there holds for all piecewise smooth $\bv\in \bH^1(\Omega_h^\circ)$,
\[
\theta_z(\Div \bv) = 0\qquad \forall z\in \calS_h^{\circ,I}.
\]
\end{proposition}

With an abuse of notation, for each $T\in \calT_h$, 
we set $T^{ct}$ to be the set of three triangles obtained by
connecting the vertices of $T$ with its incenter.
We also define $T^{ps}$ to be the local Powell-Sabin refinement of $T$, i.e.,
\[
T^{ps} = \{K\in \calT_h^{ps}:\ K\subset \bar T\}.
\]
For $T\in \calT_h$, let $z\in \calS_h$ be
a singular vertex with $z\in \bar T$.  Let $\{K_1,K_2\}\subset T^{ps}$
be the set of two triangles that have $z$ as a vertex, 
and let ${\bm \eta}_i$ be the outward unit normal of $\p K_i$
orthogonal to the common edge $\p K_1\cap \p K_2$; see Figure \ref{fig:PSonT}.
We define the jump of a piecewise smooth function $q$ on $T$
at $z$ as
\[
\jump{q}(z) = q|_{K_1}(z) {\bm \eta}_1+ q|_{K_2}(z) {\bm \eta}_2.
\]
Thus, we see that a piecewise smooth function $q$ satisfies
$\theta_z(q) = 0$ if and only if $[q](z)$
is single-valued.

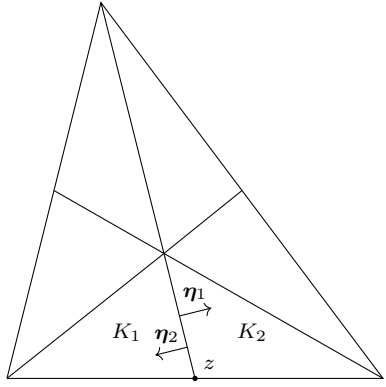
\begin{figure}[h]
\centering
\begin{tikzpicture}[scale=1.25]
\draw[-](0,0)--(4,0)--(1,4)--(0,0);
\node (z1) at (1.67,1.33) {}; 

\draw[-](0,0)--(1.67,1.33);
\draw[-](4,0)--(1.67,1.33);
\draw[-](1,4)--(1.67,1.33);
\draw[-](1.67,1.33)--(2,0);
\draw[-](1.67,1.33)--(2.5,2);
\draw[-](1.67,1.33)--(0.5,2);

\node[inner sep = 0pt,minimum size=2pt,fill=black!100,circle] (n2) at (2,0)  {};
\draw (2.15,0.15) node {\footnotesize $z$};

\draw (1.273333333,0.4933333333) node {\footnotesize $K_1$};
\draw (2.606666667,0.4933333333) node {\footnotesize $K_2$};

\draw[->](1.835000000,0.6650000000)--(2.174699586,0.7492863634);

\draw[->](1.917500000,0.3325)--(1.577800414,0.2482136365);

\draw (2,0.9) node {\footnotesize ${\bm \eta}_1$};
\draw (1.7,0.45) node {\footnotesize ${\bm \eta}_2$};

\end{tikzpicture}
\caption{\label{fig:PSonT}Depiction of a local Powell-Sabin triangulation $T^{ps}$.}
\end{figure}

\subsection{Finite element method on Powell-Sabin splits}
We define the finite element spaces
\begin{align*}
\bV_h^{ps} &= \{\bv\in \bpol_1^c(\calT_h^{ps,e}),\  \int_{\p \Omega_h^i} \bv\cdot \bn = 0\},\quad
Q_h^{ps}= \{q\in \pol_0(\calT_h^{ps,e}),\ \theta_z(q) = 0\ \forall z\in \calS_h^{e,I},\ q|_{\Omega_h^i} \in L^2_0(\Omega_h^i)\},
\end{align*}
and the corresponding spaces with respect to the interior mesh:
\begin{align*}
\bV_h^{ps,i} &= \mathring{\bpol}_1^c(\calT_h^{ps,i}),\qquad
Q_h^{ps,i} = \{q\in \mathring{\pol}_0(\calT_h^{ps,i}),\ \theta_z(q) = 0\ \forall z\in \calS^i_h\}.
\end{align*}
\begin{remark}
There holds $\Div \bV_h^{ps,i} = Q_h^{ps,i}$ \cite{GuzEtal20,FabEtAl22}.
\end{remark}

We consider the analogous finite element method
of \eqref{fem}, but defined on Powell-Sabin splits:
find
$(\bu_h,p_h)\in \bV^{ps}_h\times Q^{ps}_h$ such that
\begin{align}\label{eqn:PSfem}
\Big\{\begin{array}{ll}
a_h(\bu_h,\bv_h)+b(p_h,\bv_h)=(\bm f,\bv_h),\\
b(q_h,\bu_h)-{\frac{1}{1+\gamma}}J_h(p_h,q_h)=0
\end{array}
\end{align}
for all $(\bv_h,q_h)\in\bV^{ps}_h \times Q^{ps}_h$. Here, the bilinear form
$a_h(\cdot,\cdot)$ is given by \eqref{eqn:ahDef} but with 
\begin{align*}
j_h(\bu,\bv) = \sum_{K\in \calT_h^{ps,\Gamma}} \frac1{h_K} \int_{K_\Gamma} \bu\cdot \bv,\qquad
{\bf j}_h(\bu,\bv) = \sum_{F\in \calF_h^{ps,\Gamma}} h_F \int_F [\p_n \bu][\p_n \bv],
\end{align*}
where $\mcf^{ps,\Gamma}$ is the set of edges in $\calT_h^{ps,\Gamma}$
that do not lie on $\p\Omega_h^{e}$.
The bilinear form $b(\cdot,\cdot)$ is defined in \eqref{eqn:bDef},
and the pressure ghost-stabilization term is
\[
J_h(q,p)=\sum_{F\in\mcf^{ps,\Gamma}} h_F \int_{F}[q][p].
\]
\begin{lemma}
With a slight abuse of notation, let $\calT_h^{ct,e}$
be the Clough-Tocher refinement of { $\calT_h^e$}
obtained by connecting the vertices of each triangle with its incenter.
Likewise, let $\widetilde{\mathcal{T}}_h^{ct,i}$ be defined
as in \eqref{eqn:TildeThCT} but with incenter refinement, and 
set
\[
\widetilde \Omega_h^{ct,i} = {\rm Int}\Big(\bigcup_{T\in \widetilde{\mathcal{T}}_h^{ct,i}} \bar K\Big).
\]
Then if $\bu\in \bV_h^{ps}$ satisfies \eqref{eqn:PSfem}, there holds $\Div \bu_h =0$ in $\widetilde \Omega_h^{ct,i} $.
\end{lemma}
\begin{proof}
For any $q\in \pol_0(\calT_h^{ct,e})$ with $q|_{\Omega_h^i}\in L^2_0(\Omega_h^i)$,
there holds $q\in Q^{ps}_h$.  Consequently, we can apply steps (1)--(3) in the proof of 
Lemma \ref{L:divfree} verbatim to conclude
\begin{equation}\label{eqn:LocalPSDiv}
\int_K \Div \bu_h = 0\qquad \forall K\in \widetilde \calT_h^{ct,i}.
\end{equation}
Next, fix a $K\in \widetilde\calT_h^{ct,i}$, and let $z\in \calS_h^{i,I}$
be the singular vertex in $\calT_h^{ps,i}$ such that $z\in \bar K$.
Let $\{K_z^{(1)},K_z^{(2)},K_z^{(3)},K_z^{(4)}\}\subset \calT_h^{ps,i}$
be the triangles in the Powell-Sabin refinement 
that have $z$ as a vertex. 
Let $K_\dagger \in \widetilde \calT_h^{ct,i}$ be an arbitrary triangle satisfying
 $K_\dagger\cap K_z^{(j)} = \emptyset\ (j=1,2,3,4)$, and set
 \begin{align*}
 q = \left\{
 \begin{array}{ll}
 \Div \bu_h & \text{on }K_z^{(j)}\ j=1,2,3,4,\\
 c & \text{on }K_\dagger,\\
 0 & \text{otherwise},
 \end{array} 
 \right.
 \end{align*}
 where $c\in \bbR$ is chosen such that $\int_{\Omega_h^i} q = 0$.
 By Proposition \ref{prop:DivOnPS}, there holds $q\in Q_h^{ps}$.  
We use \eqref{eqn:PSfem} and \eqref{eqn:LocalPSDiv} to obtain
\[
\sum_{j=1}^4 \int_{K_z^{(j)}} |\Div \bu_h|^2 = - c\int_{K_\dagger} \Div \bu_h = 0.
\]
Because $K\subset \cup_{j=1}^4 K_z^{(j)}$, we conclude $\Div \bu_h =0$
in $K$, and therefore $\Div \bu_h=0$ in $\widetilde \Omega_h^{ct,i}$.
\end{proof}

\subsection{Stability analysis on Powell-Sabin splits}
In this section, we derive a inf-sup condition
for the finite element pair $\bV_h^{ps,i}\times Q_h^{ps,i}$
on $\Omega_h^i$ that is uniformly bounded with respect to the discretization 
parameter $h$.  As a first step,
we state the degrees of freedom (DOFs)
of the finite element spaces given in \cite{GuzEtal20,FabEtAl22}.
\begin{lemma}\label{lem:PSDOFs}
A function $\bv\in \bV_h^{ps,i}$ is uniquely determined
by the values
\begin{alignat*}{2}
&\bv(a)\qquad  &&\text{for all vertices $a$ in $\calT_h^i$},\\
&\int_F (\bv\cdot \bn_F)\qquad &&\text{for all edges $F$ in $\calT_h^i$},\\
&\jump{\Div \bv}(z)\qquad &&\text{for all $z\in \calS_h^i$},\\
&\int_T (\Div \bv)p\qquad &&\forall p\in \mathring{\pol}_0(T^{ct})\quad \text{for all $T\in \calT_h^i$},
\end{alignat*}
where $\bn_F$ is a unit normal of $F$.
A function $q\in Q_h^{ps,i}$ is uniquely determined by
the values
\begin{alignat*}{2}
&\jump{q}(z)\qquad &&\text{for all $z\in \calS_h^i$},\\
&\int_T q p\qquad &&\forall p\in \pol_0(T^{ct})\quad \text{for all }T\in \calT_h^i.
\end{alignat*}
\end{lemma}

As an intermediate step in the stability analysis,
we first show an inf-sup stability result,
but with $Q_h^{ps,i}$ replaced
by piecewise constants with respect to the triangulation $\calT_h^i$.
\begin{lemma}\label{lem:PSPreStab}
{ There exists} $\beta_1>0$ independent of $h$ such that
\[
\beta_1\|\bar q\|_{L^2(\Omega_h^i)}\le \sup_{\bv\in \bV_h^{ps,i}\backslash \{0\}} \frac{\int_{\Omega_h^i} (\Div \bv) \bar q}{\|\nab \bv\|_{L^2(\Omega_h^i)}}\qquad 
\forall \bar q\in \mathring{\pol}_0(\calT_h^i).
\]
\end{lemma}
\begin{proof}
Let $\bar q\in  \mathring{\pol}_0(\calT_h)$.  Using the $P_2-P_0$ stability result in \cite{GuzmanMaxim18}, there exists
$\bar \bv\in \bpol_2^c(\calT_h^i)$ satisfying
\begin{equation}\label{eqn:P2P0Stab}
\beta_0\|\bar q\|_{L^2(\Omega_h^i)}\le \frac{\int_{\Omega_h^i} (\Div \bar \bv) \bar q}{\|\bar \bv\|_{H^1(\Omega_h^i)}},
\end{equation}
where $\beta_0$ is uniformly bounded below with respect to $h$.
We then use Lemma \ref{lem:PSDOFs} to uniquely define $\bv\in \bV_h^{ps,i}$ by the conditions
\begin{alignat*}{2}
&\bv(a) = \bar \bv(a)\qquad  &&\text{for all vertices $a$ in $\calT_h^i$},\\
&\int_F (\bv\cdot \bn_F) = \int_F (\bar \bv\cdot \bn_F)\qquad &&\text{for all edges $F$ in $\calT_h^i$},\\
&\jump{\Div \bv}(z) = \jump{\Div \bar \bv}(z) = 0\qquad &&\text{for all $z\in \calS_h^i$},\\
&\int_T (\Div \bv)p = \int_T (\Div \bar \bv)p\qquad &&\forall p\in \mathring{\pol}_0(T^{ct})\text{ for all }T\in \calT_h^i.
\end{alignat*}
Because $\bar q$ is piecewise constant on $\calT_h^i$, the second condition implies
\[
\int_T (\Div \bv)\bar q = \int_T (\Div \bar \bv)\bar q\qquad \forall T\in \calT_h^i,
\]
and so
\begin{align}\label{eqn:CommuteABC}
\int_{\Omega_h^i} (\Div \bv)\bar q =  \int_{\Omega_h^i} (\Div \bar \bv)\bar q.
\end{align}
A standard scaling shows $\| \bv\|_{H^1(\Omega_h^i)}\le C_0 \| \bar \bv\|_{H^1(\Omega_h^i)}$
with $C_0>0$ independent of $h$.
This estimate, along with \eqref{eqn:P2P0Stab}--\eqref{eqn:CommuteABC} implies the result
with $\beta_1 = \beta_0/C_0$.
\end{proof}

\begin{theorem}\label{thm:PSInfSup}
There  exists $\beta_*>0$ independent of $h$ such that
\begin{equation}\label{eqn:PSInfSup}
\beta_*\| q\|_{L^2(\Omega_h^i)}\le \sup_{\bv\in \bV_h^{ps,i}\backslash \{0\}} \frac{\int_{\Omega_h^i} (\Div \bv)  q}{\| \bv\|_{H^1(\Omega_h^i)}}\qquad \forall  q\in Q_h^{ps,i}.
\end{equation}
\end{theorem}
\begin{proof}
Fix $q\in Q_h^{ps,i}$, and let $\bar q\in \mathring{\pol}_0(\calT_h^i)$ be the $L^2$-projection of $q$ onto $\mathring{\pol}_0(\calT_h^i)$, i.e.,
\[
\bar q = \frac1{|T|} \int_T q \qquad \forall T\in \calT_h^i.
\]
Using Lemma  \ref{lem:PSDOFs}, we define $\bw\in \bV_h^{ps,i}$ by the conditions
\begin{alignat*}{2}
&\bw(a) = 0\qquad  &&\text{for all vertices $a$ in $\calT_h^i$},\\
&\int_F (\bw\cdot \bn_F) = 0\qquad &&\text{for all edges $F$ in $\calT_h^i$},\\
&\jump{\Div \bw}(z) = \jump{(q-\bar q)}(z) = \jump{q}(z)\qquad &&\text{for all $z\in \calS_h^i$},\\
&\int_T (\Div \bw)p = \int_T (q-\bar q)p\qquad &&\forall p\in \mathring{\pol}_0(T^{ct})\text{ for all $T\in \calT_h^i$}.
\end{alignat*}
Because $\Div \bw- (q-\bar q)\in Q_h^{ps,i}$,
the last {  two} conditions imply $\Div \bw = q-\bar q$ by Lemma \ref{lem:PSDOFs}.  Moreover, scaling shows $\| \bw\|_{H^1(\Omega_h^i)}\le C_1 \|q-\bar q\|_{L^2(\Omega_h^i)}$
with $C_1>0$ independent of $h$.
This implies
\begin{align*}
C^{-1}_1 \|q-\bar q\|_{L^2(\Omega_h^i)} \le  \frac{\int_{\Omega_h^i} (q-\bar q)(q-\bar q)}{\| \bw\|_{H^1(\Omega_h^i)}}
=  \frac{\int_{\Omega_h^i} (q-\bar q)q}{\| \bw\|_{H^1(\Omega_h^i)}}
=  \frac{\int_{\Omega_h^i} (\Div \bw)q}{\| \bw\|_{H^1(\Omega_h^i)}}
\le \sup_{\bv\in \bV_h^{ps,i}\backslash \{0\}}  \frac{\int_{\Omega_h^i} (\Div \bv)q}{\| \bv\|_{H^1(\Omega_h^i)}}.
\end{align*}
Finally, using this estimate and Lemma \ref{lem:PSPreStab} we conclude
\begin{align*}
\|q\|_{L^2(\Omega_h^i)}
&\le \|q-\bar q\|_{L^2(\Omega_h^i)} + \|\bar q\|_{L^2(\Omega_h^i)}\\
&\le \|q-\bar q\|_{L^2(\Omega_h^i)} + \beta_1^{-1}\Big(\sup_{\bv\in \bV_h^{ps,i}\backslash \{0\}} \frac{\int_{\Omega_h^i} (\Div \bv)q}{\| \bv\|_{H^1(\Omega_h^i)}} + \|q-\bar q\|_{L^2(\Omega_h^i)}\Big)\\
&\le \Big(C_1(1+\beta_1^{-1}) +\beta_1^{-1}\Big)
  \sup_{\bv\in \bV_h^{ps,i}\backslash\{0\}} \frac{\int_{\Omega_h^i} (\Div \bv)q}{\| \bv\|_{H^1(\Omega_h^i)}}.
\end{align*}
Thus, \eqref{eqn:PSInfSup} holds with $\beta_* = \Big(C_1(1+\beta_1^{-1}) +\beta_1^{-1}\Big)^{-1}$.
\end{proof}

From the inf-sup stability result
in Theorem \ref{thm:PSInfSup},  we obtain the following stability 
result for the finite element method \eqref{eqn:PSfem}.  Since its proof is essentially the same
as the proof of Theorem \ref{thm:WP}, it is omitted.

\begin{theorem}\label{thm:WPofPS}
There exists a unique $(\bu_h,p_h)\in\bV^{ps}_h\times Q^{ps}_h$ satisfying \eqref{eqn:PSfem}.
Moreover,
\begin{align*}
\| p_h\|_{L^2(\Omega)}
\leq C(1+\gamma)^{\frac12}\|{\bm f}\|_{\bV_{h}'},\quad
\|\bu_h\|_{\bV_h}\leq C\|{\bm f}\|_{\bV_h'},
\end{align*}
for some $C>0$ independent of $\gamma$, $h$, and the position of $\Gamma$ in the mesh.
\end{theorem}

\subsection{Convergence analysis on Powell-Sabin splits}
Here, we adopt the arguments
of Section \ref{sec:Conv} to the 
finite element method \eqref{eqn:PSfem}
defined on Powell-Sabin splits.
The key result is the following lemma which establishes
the approximation properties
of the discrete divergence-free subspace.
\begin{lemma}
Let 
\[
\bZ_h^{ps} = \{\bv\in \bV^{ps}_h:\ \Div \bv = 0\text{ in }\Omega_h^e\},
\]
and let $\bu$ be the divergence-free extension of the solution to \eqref{eqn:Stokes}.
There holds
\[
\inf_{\bv\in \bZ_h } \|\nab (\bu-\bv)\|_{H^1(T)}\le C h_T |\bu|_{H^2(\omega_T)}\qquad \forall T\in \calT_h^e.
\]
\end{lemma}
\begin{proof}
Let ${\bm I}_h \bu\in \bpol_1^c(\calT_h^e)$ denote the linear Scott-Zhang interpolant of $\bu$
with respect to $\calT_h^e$.  We define
define $\bv\in \bV_h^{ps}$ uniquely via the conditions
\begin{alignat*}{2}
&\bv(a) = ({\bm I}_h\bu)(a) \qquad  &&\text{for all vertices $a$ in $\calT_h^e$},\\
&\int_F (\bv\cdot \bn_F) = \int_F (\bu\cdot \bn_F)\qquad &&\text{for all edges $F$ in $\calT_h^e$},\\
&\jump{\Div \bv}(z) = 0\qquad &&\text{for all $z\in \calS_h^e$},\\
&\int_T (\Div \bv)p = 0&&\forall p\in \mathring{\pol}_0(T^{ct})\text{ for all $T\in \calT_h^e$}.
\end{alignat*}
Using the last three conditions
and the inclusion $\Div \bv\in Q_h^{ps}$,
we conclude $\Div \bv = 0$ by Lemma~\ref{lem:PSDOFs}.
Furthermore, noting $(\bv-\bI_h \bu)\in \bV_h^{ps}$, 
there holds by scaling and properties of the Scott-Zhang interpolant, 
\begin{align*}
\|\nab (\bv-{\bm I}_h \bu)\|_{L^2(T)}^2 
&\le C \Big(h_T^{-2}\Big|\int_{\p T} (\bv-\bI_h \bu)\cdot \bn\Big|^2
+h_T^2 \mathop{\sum_{z\in \calS_h^e}}_{z\in \bar T} \big|\jump{\Div (\bv-\bI_h \bu)}(z)\big|^2\\
&\qquad +\mathop{\sup_{p\in \mathring{\pol}_0(T^{ct})}}_{\|p\|_{L^2(T)}=1} \Big|\int_T (\Div(\bv-\bI_h \bu))p\Big|^2\Big)\\
&= C \Big(h_T^{-2}\Big|\int_{\p T} (\bu-\bI_h \bu)\cdot \bn\Big|^2
 +\mathop{\sup_{p\in \mathring{\pol}_0(T^{ct})}}_{\|p\|_{L^2(T)}=1} \Big|\int_T \Div(\bu-\bI_h \bu)p\Big|^2\Big)\\ 
&\le  C\big( h_T^{-1} \|\bu-{\bm I}_h \bu\|_{L^2(\p T)}^2+\|\nab (\bu-\bI_h \bu)\|_{L^2(T)}^2)\\
&\le C \big(\|\nab (\bu-{\bm I}_h \bu)\|_{L^2(T)}^2 + h_T^{-2} \|\bu - {\bm I}_h \bu\|_{L^2(T)}^2\big)\le C h_T^2 |\bu|_{H^2(\omega_T)}^2.
\end{align*}
Thus,
\[
\|\nab (\bu-\bv)\|_{L^2(T)}\le \|\nab (\bu-{\bm I}_h \bu)\|_{L^2(\Omega)}+\|\nab (\bv-{\bm I}_h \bu)\|_{L^2(T)}\le C h_T |\bu|_{H^2(\omega_T)}.
\]
\end{proof}

With the approximation properties
of the divergence-free subspace established,
we can use the same arguments
in the proof of Theorem \ref{thm:MainError} (with $k=1$)
to derive a first-order error estimate.
\begin{theorem}
Let $(\bu_h,p_h)\in \bV_h^{ps}\times Q_h^{ps}$
be the solution to \eqref{eqn:PSfem}.
Then the estimate \eqref{VelErr} holds with $k=1$.
\end{theorem}

\section{Numerical Experiments}\label{sec:Num}

In this section, we perform some simple numerical experiments
and compare the results with the theory developed in the previous sections.
In the set of experiments, we
take the domain to be the circle with center
$(0.5,0.5)$ and radius $\sqrt{0.2}$:
\[
\Omega = \{x\in \bbR^2:\ (x_1-0.5)^2+(x_2-0.5)^2<0.2\}.
\]
The data is chosen such that the exact solution is 
\begin{equation}\label{eqn:ExactSolution}
\bu = 
\begin{pmatrix}
 2(x_1^2-x_1+0.25+x_2^2-x_2)(2x_2-1)\\
 -2(x_1^2-x_1+0.25+x_2^2-x_2)(2x_1-1)
 \end{pmatrix}
,\qquad
p = 10^3 \big(10(x_1^2-x_2^2)^2+c\big),
\end{equation}
with normalizing constant $c\in \bbR$.
We take the covering domain $S = (0,1)^2$, the unit square,
and consider a sequence of type I triangulations $\calT_h$
defined on $S$.

We compute the finite element method \eqref{fem}
\rev{on Clough-Tocher splits}
with $k=2$
for a decreasing set of mesh parameters
$h$ and various grad-div parameters $\gamma$.
Approximate numerical integration rules on cut elements 
and on the boundary $\Gamma$ are defined 
using isoparametric mappings similar to \cite{Lehrenfeld16}.

The resulting $H^1$ velocity errors
and $L^2$ pressure errors are given
in Figure  \ref{fig:beta1E3Errors},
and the $L^2$ divergence error of the computed velocities
are presented in Figure \ref{fig:beta1E3ErrorsDivergence}.
These error plots show asymptotic optimal order (quadratic) convergence rates
for both the discrete velocity and pressure solutions for fixed 
grad-div parameter $\gamma$ and fixed Nitsche parameter $\eta$.
The figures also list the errors
with grad-div parameter $\gamma = 10h^{-1}$ 
for both constant Nitsche parameter $\eta=100$
and mesh-dependent Nitsche parameter $\eta = 10h^{-1}$.
The results indicate that the method performs
best, with respect to errors,
if both penalty parameters scale like $O(h^{-1})$.

\begin{figure}
\centering
\begin{tikzpicture}[scale = 0.9]
\begin{loglogaxis}[
	xlabel={$h$},
	ylabel={$\|\nab (\bu - \bu_h)\|_{L^2(\Omega)}$},
	legend style={nodes={scale=0.6, transform shape}},
	legend pos=north west]
\addplot[color = blue] coordinates {
(0.2,1.79E+00)
(0.1,5.00E-01)
(0.05,1.17E-01)
(0.025,2.22E-02)
(0.0125,3.91E-03)
(0.00625,7.30E-04)
};
\addplot[color = red,dashed] coordinates {
(0.2,1.74E+00)
(0.1,4.88E-01)
(0.05,1.14E-01)
(0.025,2.16E-02)
(0.0125,3.81E-03)
(0.00625,7.13E-04)
};

\addplot[color = purple,dotted] coordinates {
(0.2,9.85E-01)
(0.1,6.98E-01)
(0.05,1.15E-01)
(0.025,1.51E-02)
(0.0125,7.38E-03)
(0.00625,1.60E-03)
};
\addplot[color = black,dashed] coordinates {
(0.2,9.95E-01)
(0.1,6.63E-01)
(0.05,1.33E-01)
(0.025,6.53E-03)
(0.0125,9.59E-04)
(0.00625,2.18E-04)
};

\legend{{$\gamma=0,\eta=100$},{$\gamma=1,\eta=100$},{$\gamma=10h^{-1},\eta=100$},{$\gamma=10h^{-1},\eta=10h^{-1}$}};
\end{loglogaxis}
\end{tikzpicture}
\begin{tikzpicture}[scale = 0.9]
\begin{loglogaxis}[
	xlabel={$h$},
	ylabel={$\|p - p_h\|_{L^2(\Omega)}$},
	legend style={nodes={scale=0.6, transform shape}},
	legend pos=north west]
\addplot[color = blue] coordinates {
(0.2,4.65E+01)
(0.1,1.34E+01)
(0.05,3.36E+00)
(0.025,7.73E-01)
(0.0125,1.96E-01)
(0.00625,4.79E-02)
};
\addplot[color = red,dashed] coordinates {
(0.2,4.67E+01)
(0.1,1.34E+01)
(0.05,3.38E+00)
(0.025,7.74E-01)
(0.0125,1.96E-01)
(0.00625,4.79E-02)
};

\addplot[color = purple,dotted] coordinates {
(0.2,6.14E+01)
(0.1,5.04E+01)
(0.05,1.34E+01)
(0.025,1.76E+00)
(0.0125,1.57E+00)
(0.00625,5.60E-01)
};
\addplot[color = black,dashed] coordinates {
(0.2,6.19E+01)
(0.1,4.80E+01)
(0.05,1.52E+01)
(0.025,1.15E+00)
(0.0125,2.68E-01)
(0.00625,7.97E-02)
};

\legend{{$\gamma=0,\eta=100$},{$\gamma=1,\eta=100$},{$\gamma=10h^{-1},\eta=100$},{$\gamma=10h^{-1},\eta=10h^{-1}$}};
\end{loglogaxis}
\end{tikzpicture}
\caption{\label{fig:beta1E3Errors}Errors on a sequence of refined triangulations
for the velocity (left) and pressure (right) with different grad-div parameters.}
\end{figure}
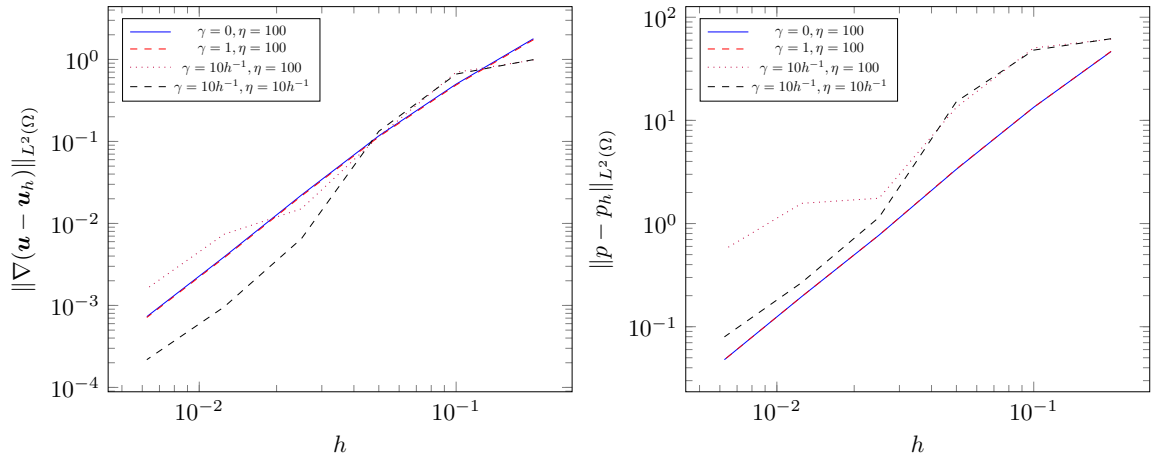

\begin{figure}
\centering
\begin{tikzpicture}[scale = 0.9]
\begin{loglogaxis}[
	xlabel={$h$},
	ylabel={$\|\Div\bu_h\|_{L^2(\Omega)}$},
	legend style={nodes={scale=0.6, transform shape}},
	legend pos=north west]
\addplot[color = blue] coordinates {
(0.2,1.18E+00)
(0.1,3.15E-01)
(0.05,8.36E-02)
(0.025,1.49E-02)
(0.0125,2.49E-03)
(0.00625,4.52E-04)
};

\addplot[color = red,dashed] coordinates {
(0.2,1.14E+00)
(0.1,3.05E-01)
(0.05,8.06E-02)
(0.025,1.43E-02)
(0.0125,2.41E-03)
(0.00625,4.38E-04)
};

\addplot[color = purple,dotted] coordinates {
(0.2,4.27E-01)
(0.1,2.43E-01)
(0.05,3.47E-02)
(0.025,2.46E-03)
(0.0125,1.19E-03)
(0.00625,2.21E-04)
};

\addplot[color = black,dashed] coordinates {
(0.2,4.36E-01)
(0.1,2.30E-01)
(0.05,3.93E-02)
(0.025,1.37E-03)
(0.0125,1.47E-04)
(0.00625,2.65E-05)
};
\legend{{$\gamma=0,\eta=100$},{$\gamma=1,\eta=100$},{$\gamma=10h^{-1},\eta=100$},{$\gamma=10h^{-1},\eta=10h^{-1}$}};
\end{loglogaxis}
\end{tikzpicture}
\includegraphics[scale=0.14]{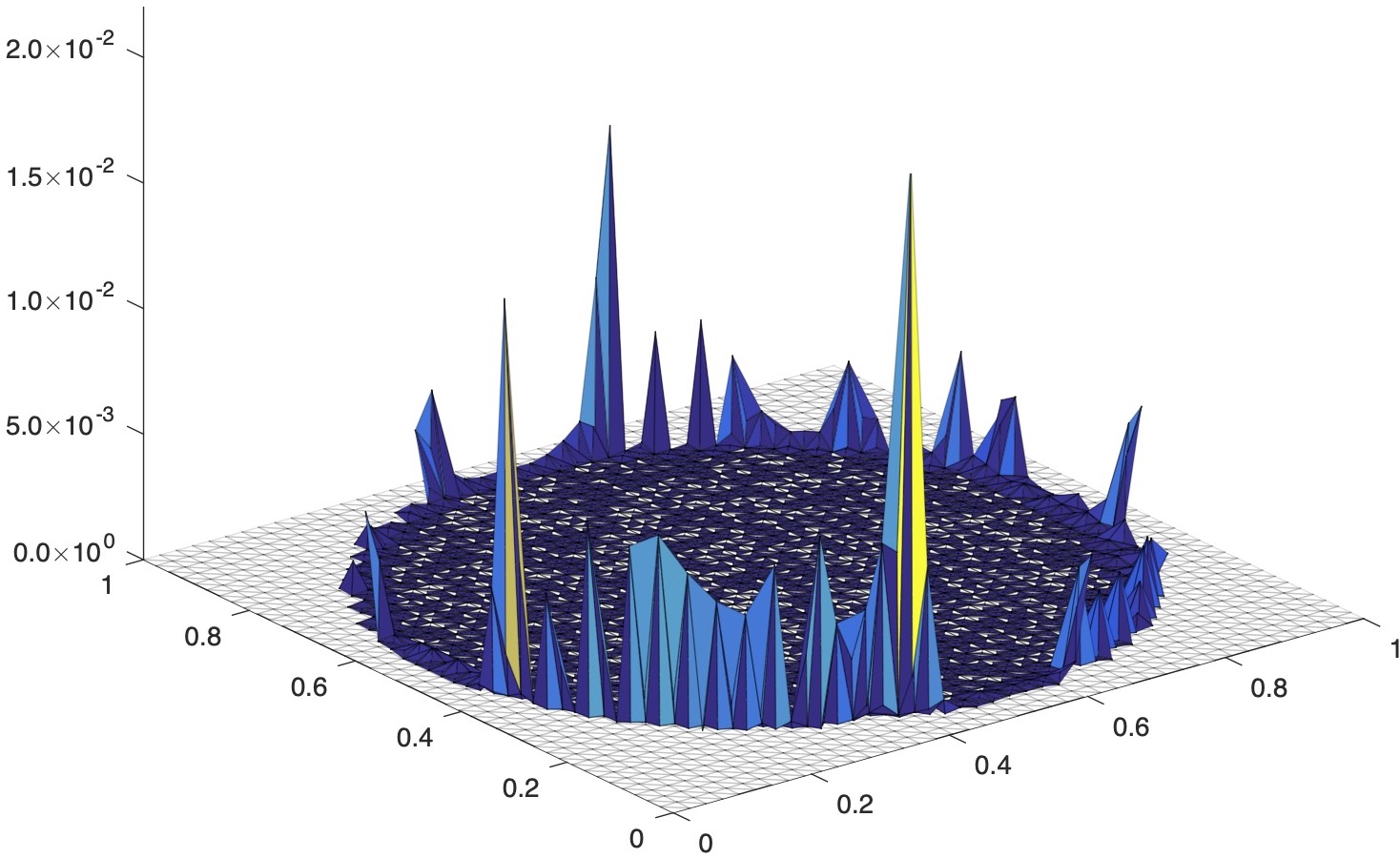}
\caption{\label{fig:beta1E3ErrorsDivergence} Left: Divergence errors on a sequence of refined triangulations
for the velocity with different grad-div parameters. Right: Plot of $|\Div \bu_h|$ with $h=1/40$, $\gamma = \eta = 10h^{-1}$.
}

\end{figure}
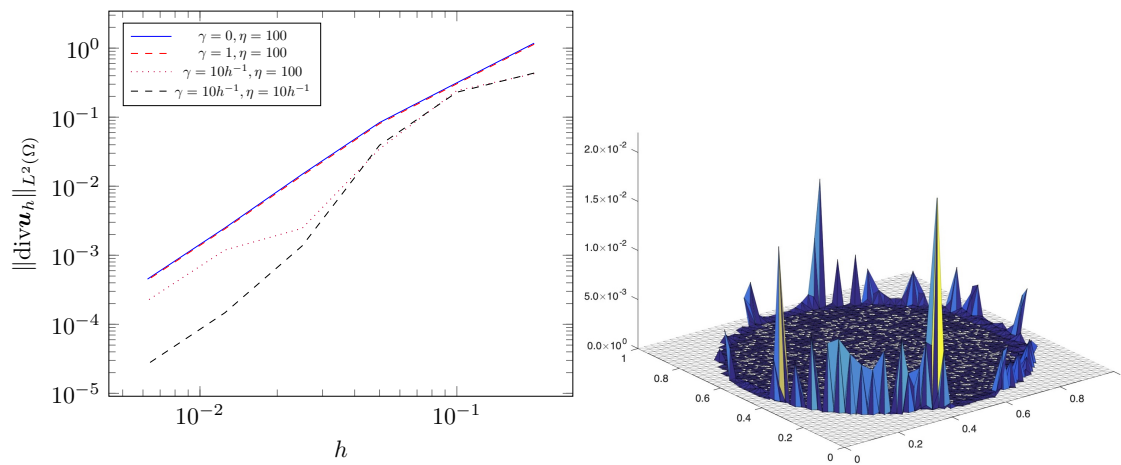

\appendix

\section{Proof of \eqref{ApproxProp1}}
The approximation property \eqref{ApproxProp1} is derived
by constructing a Fortin operator using the recent results in \cite{FGN20}.
We consider the three-dimensional case; the analogous 2D arguments are similar (and simpler).

For $T\in \calT_h^e$, let $T^{ct}$ denote the local triangulation of four (sub)tetrahedra, obtained by performing a Clough-Tocher (or Alfeld)
split of $T$.  
We also define the smooth space
\[
\mathring{\bM}_{k+1}(T^{ct}) = \{\bkappa\in \mathring{\bpol}_{k+1}^c(T^{ct}): {\bf curl}\,\bkappa\in \mathring{\bpol}_k^c(T^{ct})\}.
\]

\begin{lemma}
There exists an operator
$\bPi_{0,T}:\bH^1(T)\to \mathring{\bpol}_k^c(T^{ct})$ uniquely determined by the conditions
\begin{subequations}
\label{eqn:ConditionsPi0}
\begin{alignat}{3}
&\int_T (\bPi_{0,T} \bv)\cdot {\bf curl}\,\bkappa = \int_T \bv \cdot {\bf curl}\,\bkappa \qquad &&\forall \bkappa\in \mathring{\bM}_{k+1}(T^{ct}),\\
&\int_T (\Div (\bPi_{0,T} \bv)) \kappa = \int_T (\Div \bv)\kappa \qquad &&\forall \kappa\in \mathring{\pol}_{k-1}(T^{ct}).
\end{alignat}
\end{subequations}
Moreover, 
\[
\|\nab \bPi_{0,T} \bv\|_{L^2(T)}\le C\big(\|\nab \bv\|_{L^2(T)}+ h_T^{-1} \|\bv\|_{L^2(T)}\big)\qquad \forall \bv\in \bH^1(T).
\]
\end{lemma}
\begin{proof}
The existence of such an operator uniquely determined
by \eqref{eqn:ConditionsPi0} follows from \cite[Lemma 4.16]{FGN20}.

{It remains to show the stability estimate.  This is done via a scaling argument.

To ease presentation, set $\bv_T = \bPi_{0,T}\bv$.
Let $\hat T = \{\frac{x}{h_T}:\ x\in T\}$ be a dilation of $T$, 
and define $\hat \bv_T\in \mathring{\bpol}_k(\hat T^{ct})$ as
$\hat \bv_T(\hat x) = \bv_T(x)$ with $x = h_T \hat x$, so that
$\p \bv_T = h_T^{-1}\hat \p \hat \bv_T$.
In particular, $\Div \bv_T = h_T^{-1} \widehat \Div \hat \bv_T$
and $\bcurl\, \bkappa = h_T^{-1} \widehat \bcurl\,\hat \bkappa$.
Using a change of variables and equivalence of norms, we compute
\begin{align*}
h_T^{-1}\|\nab \bv_T\|_{L^2(T)}^2 
&= \|\hat \nab \hat \bv_T\|_{L^2(\hat T)}^2
\approx
\sup_{\hat \bkappa\in \mathring{\bM}_{k+1}(\hat T^{ct})\backslash \{0\}}\Big|  \frac{\int_{\hat T} \hat \bv_T\cdot \widehat \bcurl\,\hat \bkappa}{\|\widehat\bcurl \hat \bkappa\|_{L^2(\hat T)}}\Big|^2
+\sup_{\hat \kappa\in \mathring{\pol}_{k-1}(\hat T^{ct})\backslash \{0\}}\Big|  \frac{\int_{\hat T} \widehat \Div \hat \bv_T \hat \kappa}{\| \hat \kappa\|_{L^2(\hat T)}}\Big|^2\\
& = \sup_{\hat \bkappa\in \mathring{\bM}_{k+1}( T^{ct})\backslash \{0\}}\Big|  \frac{h_T^{-3} \int_{T}  \bv_T\cdot (h_T \bcurl\, \bkappa)}{h_T^{-1/2} \|\bcurl  \bkappa\|_{L^2( T)}}\Big|^2
+\sup_{ \kappa\in \mathring{\pol}_{k-1}( T^{ct})\backslash \{0\}}\Big|  \frac{h_T^{-3} \int_{ T}  (h_T \Div  \bv_T)  \kappa}{h_T^{-3/2} \|  \kappa\|_{L^2( T)}}\Big|^2\\
& = h_T^{-3} \sup_{\hat \bkappa\in \mathring{\bM}_{k+1}( T^{ct})\backslash \{0\}}\Big|  \frac{ \int_{T}  \bv\cdot  \bcurl\, \bkappa}{\|\bcurl  \bkappa\|_{L^2( T)}}\Big|^2
+h_T^{-1} \sup_{ \kappa\in \mathring{\pol}_{k-1}( T^{ct})\backslash \{0\}}\Big|  \frac{ \int_{ T}  (\Div  \bv)  \kappa}{ \|  \kappa\|_{L^2( T)}}\Big|^2\\
&\le h_T^{-3} \|\bv\|_{L^2(T)}^2+ h_T^{-1} \|\Div \bv\|_{L^2(T)}^2.
\end{align*}
}

%
%
\end{proof}

Set $\bPi_0:\bH^1(\Omega^e)\to \bV_h$ such that $\bPi_0|_T = \bPi_{0,T}$ for all $T\in \calT_h^e$.
Let $\bI_h:\bH^1(\Omega)\to \bpol^c_k(\calT_h^e)\subset \bV_h$ be the $k$th degree Scott-Zhang interpolant which satisfies ($k\ge 3$) \cite{ScottZhang90}
\begin{equation}\label{eqn:P3P0Stability}
\int_F \bI_h \bv = \int_F \bv\qquad \text{for all faces in $\calT_h^e$}\qquad \forall \bv\in \bH^1(\Omega_h^e).
\end{equation}
Finally, we set $\bPi_h:\bH^1(\Omega_h^e)\to \bV_h$ as
\[
\bPi_h = \bI_h+ \bPi_0({\bm 1}-\bI_h),
\]
where ${\bm 1}$ is the identity operator.
\begin{proposition}
There holds, for all $\bv\in \bH^1(\Omega^e)$,
\begin{align*}
\int_{\Omega^e_h} (\Div (\bPi_h\bv))q = \int_{\Omega^e_h} (\Div \bv)q\qquad \forall q\in \pol_{k-1}(\calT_h^{ct,e}).
\end{align*}
\end{proposition}
\begin{proof}
Fix $q\in \pol_{k-1}(\calT_h^{ct,e})$, and let $\bar q\in \pol_0(\calT^e_h)$
be its $L^2$ projection onto $\pol_0(\calT^e_h)$.  Note that $(q-\bar q)|_T\in \mathring{\pol}_{k-1}(T^{ct})$ for all $T\in \calT_h^e$.
We then write, using the divergence theorem, \eqref{eqn:ConditionsPi0} and \eqref{eqn:P3P0Stability},
\begin{align*}
\int_{\Omega^e} (\Div (\bPi_h\bv))q 
& = \int_{\Omega^e_h} (\Div (\bI_h\bv))(q-\bar q) + \int_{\Omega^e_h} \Div(\bPi_0({\bm 1}-\bI_h)\bv)(q-\bar q) + \int_{\Omega^e_h} (\Div (\bI_h \bv))\bar q\\
& = \int_{\Omega^e_h} (\Div (\bI_h\bv))(q-\bar q) + \int_{\Omega^e_h} \Div(({\bm 1}-\bI_h)\bv)(q-\bar q) + \int_{\Omega^e_h} (\Div  \bv)\bar q\\
& = \int_{\Omega^e_h} (\Div \bv)q.
\end{align*}
\end{proof}

The estimate \eqref{ApproxProp1} now follows from the following theorem.
\begin{theorem}
There holds for all divergence--free $\bu\in \bH^1(\Omega_h^e)$
\begin{align*}
\inf_{\bw_h\in \bZ_h} \|\nab (\bu-\bw_h)\|_{L^2(T)}\le C (\|\nab (\bu-\bI_h \bu)\|_{L^2(T)}+h_T^{-1}\|\bu-\bI_h \bu\|_{L^2(T)})\quad \forall T\in \calT_h^e.
\end{align*}
Therefore, if in addition $\bu\in \bH^{k+1}(\Omega_h^e)$, then
\begin{align*}
\inf_{\bw_h\in \bZ_h} \|\nab (\bu-\bw_h)\|_{L^2(T)}\le C h_T^{k} |\bu|_{H^{k+1}(\omega_T)}.
\end{align*}
\end{theorem}
\begin{proof}
If $\bu$ is divergence--free, then $\bPi_h \bu\in \bZ_h$.  Therefore by the definition of $\bPi_h$ and the $\bH^1$-stability of this operator,
\begin{align*}
\inf_{\bw_h\in \bZ_h} \|\nab (\bu-\bw_h)\|_{L^2(T)}
&\le  \|\nab (\bu-\bPi_h\bu)\|_{L^2(T)}\\
&\le \|\nab (\bu-\bI_h \bu)\|_{L^2(T)}+ \|\nab (\bPi_0({\bm 1}-\bI_h)\bu)\|_{L^2(T)}\\
&\le C (\|\nab (\bu-\bI_h \bu)\|_{L^2(T)}+h_T^{-1}\|\bu-\bI_h \bu\|_{L^2(T)}).
\end{align*}
We then use the approximation properties of the Scott-Zhang interpolant to obtain the result.
\end{proof}

\end{document}